\documentclass[11pt, a4paper]{amsart}
\usepackage{amscd,amssymb,amsmath,enumerate}
\usepackage[all]{xy}
\usepackage{graphics}
\usepackage[pdftex]{graphicx}
\usepackage{latexsym}
\usepackage[all]{xy}
\usepackage{psfrag}
\xyoption{matrix} \xyoption{arrow}
\usepackage{parskip}

\newtheorem{prop}{Proposition}[section]
\newtheorem{thm}[prop]{Theorem}
\newtheorem{lemma}[prop]{Lemma}
\newtheorem{cor}[prop]{Corollary}

\newtheorem{definition}[prop]{{Definition}}

\newtheorem{remark}[prop]{{Remark}}

\newtheorem{Example}[prop]{Example}

\newcommand{\cA}{{\mathcal A}}
\newcommand{\cB}{{\mathcal B}}

\newcommand{\cQ}{{\mathcal Q}}

\newcommand{\cM}{{\mathcal M}}

\newcommand{\cG}{{\mathcal G}}

\newcommand{\cC}{{\mathcal C}}

\newcommand{\cS}{{\mathcal S}}
\newcommand{\K}{{ K}}

\renewcommand{\aa}

\newcommand{\ov}{\overline}

\newcommand{\Hom}{\operatorname{Hom}\nolimits}

\newcommand{\Ker}{\operatorname{Ker}\nolimits}

\newcommand{\soc}{\operatorname{soc}\nolimits}
\newcommand{\rad}{\operatorname{rad}\nolimits}

\renewcommand{\dim}{\operatorname{dim}\nolimits}

\newcommand{\val}{\operatorname{val}\nolimits}

\newcommand{\Qg}{{\mathcal Q}_{\Gamma}}

\newcommand{\oc}{\operatorname{occ}\nolimits}
\newcommand{\bc}{Brauer configuration}
\newcommand{\bca}{\bc\ algebra}

\usepackage{color}
\definecolor{candyapplered}{rgb}{1.0, 0.03, 0.0}
\makeatletter
\def\thm@space@setup{%
  \thm@preskip=0.7cm \thm@postskip=0.3cm
}
\makeatother

\begin{document}

\date{today}
\title[Brauer configuration algebras]
{Brauer configuration algebras: A generalization of Brauer graph algebras}

\author[Green]{Edward L.\ Green}
\address{Edward L.\ Green, Department of
Mathematics\\ Virginia Tech\\ Blacksburg, VA 24061\\
USA}
\email{green@math.vt.edu}
\author[Schroll]{Sibylle Schroll}
\address{Sibylle Schroll\\
Department of Mathematics \\
University of Leicester \\
University Road \\
Leicester LE1 7RH \\
United Kingdom}
\email{schroll@le.ac.uk }

\subjclass[2010]{16G20, 16D50} 

\keywords{Brauer graph algebra, multiserial algebra, biserial algebra, symmetric radical cubed zero algebra, adjacency matrix of a graph}
\thanks{This work was supported through the Engineering and Physical Sciences Research Council, grant number EP/K026364/1, UK and by the University of Leicester in form of a study leave for the second author.  }

\begin{abstract} 
In this paper we introduce a generalization of a Brauer graph algebra which we call
a \bca. As with Brauer graphs and Brauer graph algebras, to each Brauer configuration, there is an associated Brauer configuration algebra.  We show that \bca s are finite dimensional symmetric algebras.  After studying and analysing structural properties
of \bc s and  \bca s, we show that a \bca\ is multiserial;  that is, its Jacobson radical
is a sum of uniserial modules whose pairwise intersection is either
zero or a simple module.   The paper ends with a detailed study of the relationship
between radical cubed zero \bca s,  symmetric matrices with non-negative integer entries, finite graphs and associated symmetric radical cubed zero algebras.
\end{abstract}

\date{\today}
\maketitle


\section*{Introduction}

The classification of algebras into finite, tame and wild representation type has led to many structural insights in the representation theory of finite dimensional algebras. Algebras have finite representation type if there are only finitely many isomorphism classes of indecomposable modules. In this case, the representation theory is usually very
well understood and these algebras often serve as first examples or test cases for new ideas and conjectures.  To cite but a few examples in this direction see for example \cite{D1,D2, Ma} or \cite{LY, Rickard}.


All other algebras have infinite representation type, and by \cite{Drozd} they are either tame or wild. If an algebra has tame representation type its representation theory usually still exhibits a certain regularity in its structure making calculations and establishing and proving conjectures often possible. One example of this is the class of 
Brauer graph algebras, which are, depending on their presentation, also known as symmetric special biserial algebras \cite{Roggenkamp, Schroll}.  
Brauer graph algebras are tame algebras and much of their representation theory is well-understood, see for example  \cite{A1, GR, Rickard, K} for a classification of derived equivalence classes,  \cite{ErdmannSkowronski1} for their Auslander-Reiten quiver, \cite{GreenSchrollSnashall} for their covering theory or \cite{AG, GreenSchrollSnashallTaillefer} for results on their Ext algebra. Recently there has been renewed interest in Brauer graph algebras, stemming from their connection with cluster theory on the  one side \cite{ L2, MarshSchroll, Schroll} and with mutation and derived equivalences on the other 
side \cite{A1, AAC,  Z1, Z2}.
 Furthermore, Brauer graph algebras naturally appear in the derived equivalence classification of self-injective algebras of finite and tame representation type, see \cite{SkowronskiSurvey} and the references within.  One reason that Brauer graph algebras are so well-studied and understood is that the combinatorial data of the underlying Brauer graph encodes much of the representation theory of a Brauer graph algebra such as for example, projective resolutions \cite{G, Roggenkamp} or information on the structure of the Auslander-Reiten quiver \cite{BS, Du}. 

 In the general setting of wild algebras, however, the picture is very different and much less is known or understood. In the cases where we do have some understanding, the wild algebras under consideration are often endowed with some additional structure such as for example in  \cite{Erdmann} where wild blocks of group algebras are considered, or in \cite{ErdmannKernerSkowronski}, or in \cite{KernerSkowronskiYamagata} where the algebras have other additional properties. 

In this paper we introduce a new class of mostly wild algebras, called \emph{Brauer configuration algebras}. These algebras have
additional structure arising from combinatorial data, called a \emph{Brauer configuration}. Brauer configuration algebras are a generalization of  Brauer graph algebras, in the sense that every Brauer graph
 is a Brauer configuration and every Brauer graph algebra  is a Brauer configuration algebra. However, unlike Brauer graph algebras, Brauer configuration algebras are in general of 
 wild representation type. 
But just as the Brauer graph encodes the representation theory of Brauer graph algebras, the expectation is that the Brauer configuration will encode the representation theory of Brauer configuration algebras.  As a first step in this direction, we show that  the Brauer configuration yields the Loewy structure of the indecomposable projective modules of a Brauer configuration algebra and that the dimension of the algebra can be directly read off the Brauer configuration (Proposition~\ref{prop-dim-result}). In fact,
surprisingly, one can show that the
radical of every finitely
generated module over a Brauer configuration algebra, more generally,
a special multiserial algebra, is the sum of
uniserial modules, the  intersection of any two being $(0)$ or a simple
module \cite{GreenSchroll}.   We note that multiserial and special multiserial algebras were first introduced in \cite{VHW}. In \cite{KY}  multiserial algebras have been studied with a focus on hereditary multiserial rings, and with a slightly differing definition multiserial algebras have been studied in \cite{BM, J, M}.
Furthermore, there are a number of other results valid for Brauer graph
algebras that generalize to Brauer configuration algebras, see \cite{GreenSchroll}, suggesting that  the
study of these rings and their representations could lead to interesting results.

Brauer configuration algebras contain another class of well-studied algebras that also have a combinatorial presentation in 
the form of a finite graph, 
namely that of symmetric algebras with radical cubed zero \cite{GreenSchroll}. Symmetric algebras with radical 
cube zero have been well-studied, see for example,  \cite{Benson, ErdmannSolberg1}. The representation type of symmetric algebras with radical cubed zero 
 has been classified in terms of the underlying graph in \cite{Benson} 
and for all but a finite number of families they are of wild representation type. In the case of infinite representation type, 
these algebras are Koszul  and thus we have some understanding of their representation theory  and their cohomology in 
particular \cite{ErdmannSolberg1}. 
In this paper we classify the Brauer configurations such that the associated Brauer configuration algebras are  what we call \emph{canonical} symmetric algebras with radical cubed zero. 

Brauer configuration algebras arise naturally in yet another context (in a different presentation). Namely, almost all
representatives of derived equivalence classes of symmetric algebras of finite representation type in the classification by Skowronski et al., see  
\cite{SkowronskiSurvey} and the references within, are in fact Brauer configuration algebras. Those that are not Brauer configuration algebras are in fact deformations of Brauer graph algebras. 

In Brauer configuraton algebras, similarly to the symmetric special biserial
 algebras,  a path (of length at least 2) is non-zero if and only if it lies in a \emph{special cycle}.  It is a consequence of the existence of the special cycles, that the projective indecomposable modules are  such that  their heart, given by the quotient of the radical modulo the socle, is a direct sum of uniserial
modules (Theorem~\ref{prop-proj}). 

 An interesting question to ask is whether there is a (partial) converse to Theorem~\ref{prop-proj}. Namely, in the case of special biserial algebras, we know from \cite{PS}  that a representation infinite algebra is special biserial if and only if it is a standard biserial self-injective algebra. One could raise the question whether a similar result holds for multiserial and  special multiserial algebras.  For example, given that a Brauer configuration algebra is a symmetric special multiserial algebra,  one could ask which symmetric multiserial algebras are Morita equivalent to Brauer configuration algebras. At this time these questions remain open and we thank the referee for pointing them out  
to us.

We will now sumarize the most important results in this paper.

A \emph{Brauer configuration} is a combinatorial data $\Gamma = \{ \Gamma_0, \Gamma_1, \mu, \mathfrak{o} \}$, where 
$\Gamma_0$ is the set of vertices of $\Gamma$, $\Gamma_1$ is a set of multisets of elements of $\Gamma_0$, $\mu$ is a function called the \emph{multiplicity} and $\mathfrak{o}$ is called the \emph{orientation} of $\Gamma$  (see section 1 for the precise definition of a Brauer configuration). Note that by a slight abuse of notation we call the elements in $\Gamma_1$ \emph{polygons}. We denote by $\Lambda_\Gamma$ the Brauer configuration algebra associated to $\Gamma$. 

First we have some structural results about Brauer configuration algebras.

{\bf Theorem A.}
\emph{
\begin{itemize}
\item[(1)] A Brauer configuration algebra is a finite dimensional symmetric algebra.
\item[(2)] Suppose $\Gamma = \Gamma_1 \cup \Gamma_2$ is a decomposition of $\Gamma$ into two disconnected Brauer configurations $\Gamma_1$ and $\Gamma_2$.
Then there is an algebra isomorphism $\Lambda_\Gamma \simeq \Lambda_{\Gamma_1}
\times \Lambda_{\Gamma_2}$ between the associated Brauer configuration algebras.
\item[(3)] The Brauer configuration algebra associated to a connected Brauer configuration is an indecomposable  algebra.
\item[(4)] A Brauer graph algebra is a Brauer configuration algebra.
\end{itemize}
}

 In \cite{VHW} multiserial  algebras have been defined. 
They are a generalisation of biserial algebras as defined in \cite{T, Fuller}.  A \emph{multiserial algebra} is defined to be a finite dimensional algebra $A$ such that,
for every indecomposable projective  left and right $A$-module $P$, $\rad(P)$ is a
sum of uniserial submodules, say $\rad(P)=\sum_{i=1}^nU_i$, for some $n$ and uniserial submodules $U_i$
of $\rad(P)$ with the property that, if $i\ne j$, $U_i\cap U_j$ is either $0$ or a simple $A$-module. The following result explores the structure of the indecomposable projective modules of a Brauer configuration algebra.

{\bf Theorem  B.} \emph{Let $\Lambda$ be a Brauer configuration algebra with Brauer configuration $\Gamma$.}
\emph{ \begin{enumerate}
\item There is a bijective correspondence between the set of  projective indecomposable $\Lambda$-modules and the polygons in $\Gamma$.
\item If $P$ is a projective indecomposable $\Lambda$-module corresponding to a polygon $V$ in $\Gamma$. Then   $rad(P)$ is a 
sum of $r$ indecomposable uniserial modules, where $r$ is the number of (nontruncated) vertices of $V$ and where the intersection of any two of the uniserial modules is a simple
$\Lambda$-module.
\end{enumerate}
}

 Since Brauer configuration algebras are symmetric, it then directly follows from the definition of a multiserial algebra that Brauer configuration
algebras are multiserial.

{\bf Corollary C.} \emph{A Brauer configuration algebra is a multiserial algebra.}

Finally we study  a bijective correspondence and its consequences between  finite graphs, symmetric matrices and (ordered) Brauer configurations. For this we will restrict ourselves to the case where the Brauer configuration has no self-folded polygons, or equivalently  to the case where there is no repetition of vertices in the polygons.   Before classifying  such Brauer configuration algebras with the additional condition that their Jacobson radical cubed is zero, we define  the notion of a {\it canonical } symmetric algebra with Jaobson radical cubed zero associated to a finite graph. We then show that Brauer configuration algebras with no repetition of vertices in the polygons and with radical cubed zero correspond exactly to the canonical symmetric algebras with radical cubed zero.

The paper is outlined as follows. In section~\ref{section-config} we define Brauer configurations and in section~\ref{sec-bca}
Brauer configuration algebras. Both sections contain examples to illustrate the newly defined concepts.  Section~\ref{sec-basic}
 starts out with the basic properties of Brauer configuration algebras, we then define special cycles and use these to show that the projective-injective modules of a Brauer configuration algebra are multiserial. In section~\ref{sec-r30} we define canonical symmetric algebras with radical cubed zero and relate them to
 Brauer configuration algebras with radical cubed zero and show that they correspond to exactly those Brauer configuration algebras whose Brauer configurations consist of polygons that have no self-foldings. 

\section{Brauer Configurations}\label{section-config}
In this section we define Brauer configurations which are
generalizations of Brauer graphs. 

 For the readers benefit we briefly provide
a definition of Brauer  graph algebras. Let $K$ be a field. A Brauer graph algebra is constructed from the combinatorial information contained in a
\emph{Brauer graph}, which is 
a 4-tuple $(\Gamma_0,\Gamma_1,\mu,\mathfrak o)$ where
$(\Gamma_0,\Gamma_1)$ is a finite (undirected) graph 
with vertex set $\Gamma_0$ and edge set $\Gamma_1$.  The
graph may contain
loops and multiple edges. Next, $\mu$ is a set map
from $\Gamma_0\to \mathbb N$  where $\mathbb N$	denotes
the positive integers. A vertex $\alpha\in\Gamma_0$ is called
\emph{truncated} if $\mu(\alpha)=1$ and $\alpha$ is the endpoint of
one and only one edge.  Last, $\mathfrak o$ is an orientation of $(\Gamma_0,
\Gamma_1)$; that is, at each vertex $\alpha$ in $\Gamma_0$, $\mathfrak o$ a
cyclic ordering of the edges having $v$ as one of its endpoints.  The formal definition of the orientation
and the construction a Brauer graph algebra from a Brauer
graph is a special case of the construction
of a Brauer configuration algebra from a Brauer configuration, which we give in detail below. 

 \subsection{Definition of Brauer configurations} We begin with a tuple
$\Gamma =(\Gamma_0, \Gamma_1)$, where $
\Gamma_0$ is a (finite) set of \emph{vertices} of $\Gamma$
and $\Gamma_1$ is a finite collection of labeled finite sets of
vertices where repetitions are allowed.   That is,
$\Gamma_1$ is a finite collection of finite labeled multisets whose
elements are in $\Gamma_0$.

  \begin{Example}{\rm We provide two examples which we will
continue to use throughout the paper to help clarify both definitions
and concepts.  \smallskip

 Example 1: $\Gamma=(\Gamma_0, \Gamma_1)$
with $\Gamma_0=\{1,2,3,4,5,6,7,8\}$ and $\Gamma_1=\{V_1=\{1,2,3,4,7\},
V_2=\{1,2,3,8\}, V_3=\{4,5\},V_4=\{4,6\}, V_5=\{1,4\}\}$.
\smallskip

Example 2: $\Delta=\{\Delta_0,\Delta_1\}$, with $\Delta_0=\{1,2,3,4\}$ and
$\Delta_1=\{V_1=\{1,1,1,2\}, V_2=\{1,1,3\}, V_3=\{1,2,3,4\}\}$.}
\end{Example}

We abuse notation and call the elements of $\Gamma_1$
\emph{polygons}.  The use of the term `polygon' will become
clear when we discuss realizations of configurations in section~\ref{SubsectionRealizations}.
We call the elements of a polygon $V$ the \emph{vertices of
$V$}. If $V$ is a polygon in $\Gamma_1$ and
$\alpha$ is a vertex in $\Gamma_0$, define
$\oc(\alpha,V)$ to be the number of times $\alpha$ occurs
as a vertex in $V$ and define the \emph{valence} of
$\alpha$, $\val(\alpha)$, to be
$\sum_{V\in\Gamma_1}\oc(\alpha,V)$.

\begin{Example}{\rm In  the first example above, we have $\oc(1,V_1)= 1, \oc(1,V_2)=1, \linebreak\oc(1,V_3)=0$, etc and $\val(1)=3$ and $\val(2)= 2$, etc. \\
In  the second example, we have $\oc (1,V_1)=3, \oc(1,V_2)=2,
 \oc(1,V_3)=1$, etc.  and  $\val(1)=  6$ and $\val(2)= 2$, etc.}\end{Example}

For  $\Gamma=(\Gamma_0,\Gamma_1)$ to be a
 Brauer configuration, we need two more pieces of
information.   One is a \emph{multiplicity function} $\mu\colon\Gamma_0\to \mathbb N$, where $\mathbb N$ denotes the  strictly positive integers.
The other is an \emph{orientation} $\mathfrak o$
for $\Gamma$.

\begin{Example}{\rm
In the first example, we choose $\mu(3)=3$, $\mu(5)=2$, and $\mu(i)=1$ for all other vertices.
In the second example, we   choose $\mu(i)=1$ for all vertices $i$.
Of course, there are many other choices for the multiplicity function than
the arbitrary ones given here.}\end{Example}

An \emph{orientation for $\Gamma$} is a choice,  for each vertex  $\alpha\in\Gamma_0$, of a cyclic ordering
of the polygons in which $\alpha$ occurs as a vertex,
including repetitions.
More precisely, for each vertex $\alpha\in \Gamma_0$, let $V_1,\dots, V_t$ be
the list of polygons in which $\alpha$ occurs as a vertex, with a polygon $V$
occuring  $\oc(\alpha,V)$ times in the list, that is $V$ occurs the number of times $\alpha$ occurs as a vertex
in $V$.
The cyclic order at vertex $\alpha$ is obtained by linearly ordering the list, say $V_{i_1}<\dots< V_{i_t}$ and  by adding $V_{i_t}< V_{i_1}$.
%
 Finally,
note that if $V_1<\cdots<V_t$ is the chosen cyclic ordering at
vertex $\alpha$,  then the same ordering can be represented  by any cyclic permutation such as
 $V_2<V_3<\cdots<V_t<V_1$ or  $V_3<V_4<\cdots<V_t<V_1<V_2$, etc.

\begin{Example}\label{ex-order}{\rm
In the first example, the list of polygons ocuring at vertex 1 is $V_1, V_5,V_2$, and $V_1,V_2$
at vertex 2,  $V_1, V_2$ at vertex 3, $V_1, V_3, V_4,
V_5$ at vertex 4, $V_3$ at vertex  5, $V_4$ at vertex 6, $V_1$ at vertex 7, and $V_2$ at  vertex
8. \newline
In the second example, the list of polygons ocuring at vertex 1 is \sloppy $V_1^{(1)},V_1^{(2)},
V_1^{(3)}, V_2^{(1)}$, $ V_2^{(2)}, V_3$, \sloppy where $V_1^{(1)},V_1^{(2)},
V_1^{(3)}$ are the three ocurrences of vertex $1$ in $V_1$, etc.

 An orientation is then  given by cyclically ordering the lists of polygons at each vertex.
Thus, for the first example,  one orientation would be:
$V_1<V_5<V_2$ at vertex 1, $V_1<V_2$ at vertex 2, $V_1<V_2$ at vertex 3, $V_1< V_4<V_3<
V_5$ at vertex 4,  $V_3$ at vertex 5,  $V_4$ at vertex 6, $V_1$ at vertex 7, and
$V_2$ at vertex 8.   Note that for vertices 2,4,5,6,7,8 there is only one choice for
cyclic ordering.   On the other hand,
there are 2 choices for the cyclic ordering at vertex 1 and
3! choices at vertex 4.  For later use, we call the orientation given above
$\mathfrak o_1(\Gamma)$ and let $\mathfrak o_2(\Gamma)$ be the orientation with orderings
$V_1<V_2<V_5$ at vertex 1 and $V_1<V_5<V_3<V_4$ at vertex 4.  For the remainder of
the paper, unless otherwise stated, we will use the orientation $\mathfrak o_1(\Gamma)$ when referring to example 1.

In the second example, let $\mathfrak o(\Delta)$ be the orientation given by the
orderings $V_1^{(1)}<V_1^{(2)}<
V_1^{(3)}< V_2^{(1)}< V_2^{(2)}< V_3$ at vertex 1,
 $V_1<V_3$ at vertex 2, and
$V_2<V_3$ at vertex 3. There are many other choices 
of orientations for this example and they are typically associated to non-isomorphic Brauer configuration
algebras.  }\end{Example}


\begin{definition}\label{BC} {\rm A \emph{Brauer configuration} is
a tuple $\Gamma =(\Gamma_0,\Gamma_1, \mu,\mathfrak o)$,
where $\Gamma_0$ is  a set of vertices, $\Gamma_1$ is a
set of polygons,  $\mu$ is a multiplicity function, and $\mathfrak o$
is an orientation, such that the following conditions hold.
\begin{enumerate}
\item[C1.]  Every vertex in $\Gamma_0$ is a vertex in at least
one polygon in $\Gamma_1$.
\item[C2.]  Every polygon in $\Gamma_1$ has at least two vertices.
\item[C3.]  Every polygon in $\Gamma_1$ has at least one vertex
$\alpha$ such that  $\val(\alpha)\mu(\alpha)>1$.
\end{enumerate}}
\end{definition}

Note that if $V$ and $V'$ are two distinct
polygons in $\Gamma_1$, it is possible that $V$ and $V'$ have  identical sets of vertices.
We distinguish between the two using their labels
$V$ and $V'$.   Also note that $(\{1,2\}, \{V=\{1,2\}\},\mu,\frak o)$ with $\mu(1)=\mu(2)=1$ violates C3 and hence is not a \bc. However, by convention this algebra usually is nevertheless consider to be a Brauer graph algebra and is isomorphic to the truncated polynomial algebra $k[x]/(x^2)$.

Our next goal is to show that a Brauer configuration is the union of connected Brauer configurations.
More precisely, let
$\Gamma=(\Gamma_0, \Gamma_1,\mu, \mathfrak o)$ be a Brauer configuration.
We say that $\Gamma$ is \emph{ disconnected} if there  are two Brauer configurations
$\Gamma'=(\Gamma'_0, \Gamma'_1,\mu',\mathfrak o')$ and
 $\Gamma''=(\Gamma''_0, \Gamma''_1,\mu'',\mathfrak o'')$ such that
\begin{enumerate}
\item
 $\{\Gamma'_0, \Gamma''_0\}$ is a partition of $\Gamma_0$,
\item
  for every
 polygon $V\in\Gamma_1$,
the vertices of $V$ are either all in $\Gamma_0'$ or are all in $\Gamma_0''$, 
\item   $\{\Gamma'_1,\Gamma''_1\}$ constitutes a partition of $\Gamma_1$,
\item
$\mu'$ (resp. $\mu''$) is a restriction of $\mu$ to $\Gamma'_0$ (resp. $\Gamma''_0$),   and
\item
the orientations $\mathfrak{o'}$ and $\mathfrak{o''}$ are
induced by $\mathfrak{o}$.                         
\end{enumerate}
In this case,
we write $\Gamma=\Gamma'\cup\Gamma''$.
We say \emph{$\Gamma$ is connected} if it is not disconnected.  It is clear
that any Brauer configuration can be uniquely written as a union of
connected Brauer configurations.  We call these connected Brauer configurations
the \emph{connected components of $\Gamma$}.

We say a polygon in a \bc\ is \emph{self-folded} if there is at least one vertex
which occurs more than once in $V$.   See the section below on realizations for a justification
of the terminology.  

 \begin{remark}\label{BGA} A connected Brauer configuration, all of whose polygons are
2-gons is
a Brauer graph  and a self-folded 2-gon is a loop.  For the definition
of  Brauer graphs and Brauer graph algebras, see, for example, \cite{Benson}, for a definition in terms of ribbon graphs, see \cite{MarshSchroll} or for a presentation more closely related to the present paper, see \cite{GreenSchrollSnashall}.
\end{remark}

\subsection{Realizations of Brauer configurations}\label{SubsectionRealizations}
It is very useful to visualize a Brauer configuration $\Gamma=(\Gamma_0,
\Gamma_1,\mu,\mathfrak o)$.
For this, one represents each polygon in $\Gamma_1$ by
an actual polygon.
 That is given  a polygon (or multiset) $V=\{\alpha_1,\dots,\alpha_d\}$ in $\Gamma_1$,
$V$ is visualized by an actual $d$-gon 
 with vertices
labeled by the $\alpha_i$.  Although there usually are many ways to
perform the vertex labeling, just one is chosen. In particular, the order in
which one labels the vertices of the actual  polygon is not
important.  If a vertex $\alpha\in\Gamma_0$ occurs
more than once in $V$, that is, $V$ is self-folded, we identify all  vertices labeled $\alpha$ in the actual  polygon $V$.
 Finally, we
identify vertices of different polygons if they correspond to the same
vertex in $\Gamma_0$.
We call such a choice, a \emph{realization of the configuration $\Gamma$}.
  The actual
theory and proofs in this paper never refer to or use any realization, however in terms of visualizing and understanding the results and proofs,
realizations of configurations are
a useful tool.

	\begin{Example}{\rm We provide two realizations of the first example given above.
 Here is the first realization.

\hskip 1in \begin{center} \includegraphics[scale=.5]{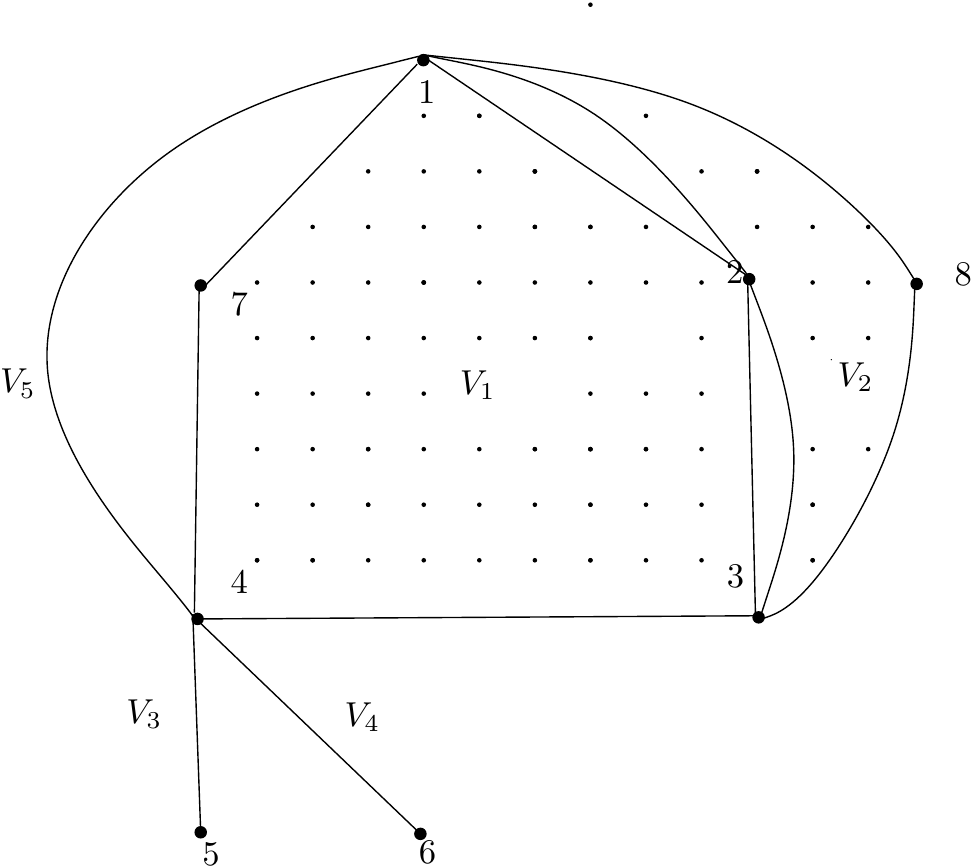} \end{center}

Note that the vertices in the polygon $V_1$ are `ordered' 1,2,3,4,7 and in the 5-gon representing
$V_1$, 1 is adjacent to 2, 2 is adjacent to 3, etc.
If we change the order in which vertices occur, we will,  in general, change the realization.   For example, in the
first example,  if instead, we ordered the vertices 1,3,4,2,7 in the realization,
we would obtain the following
\bc .

\hskip 1in \begin{center} \includegraphics[scale=.5]{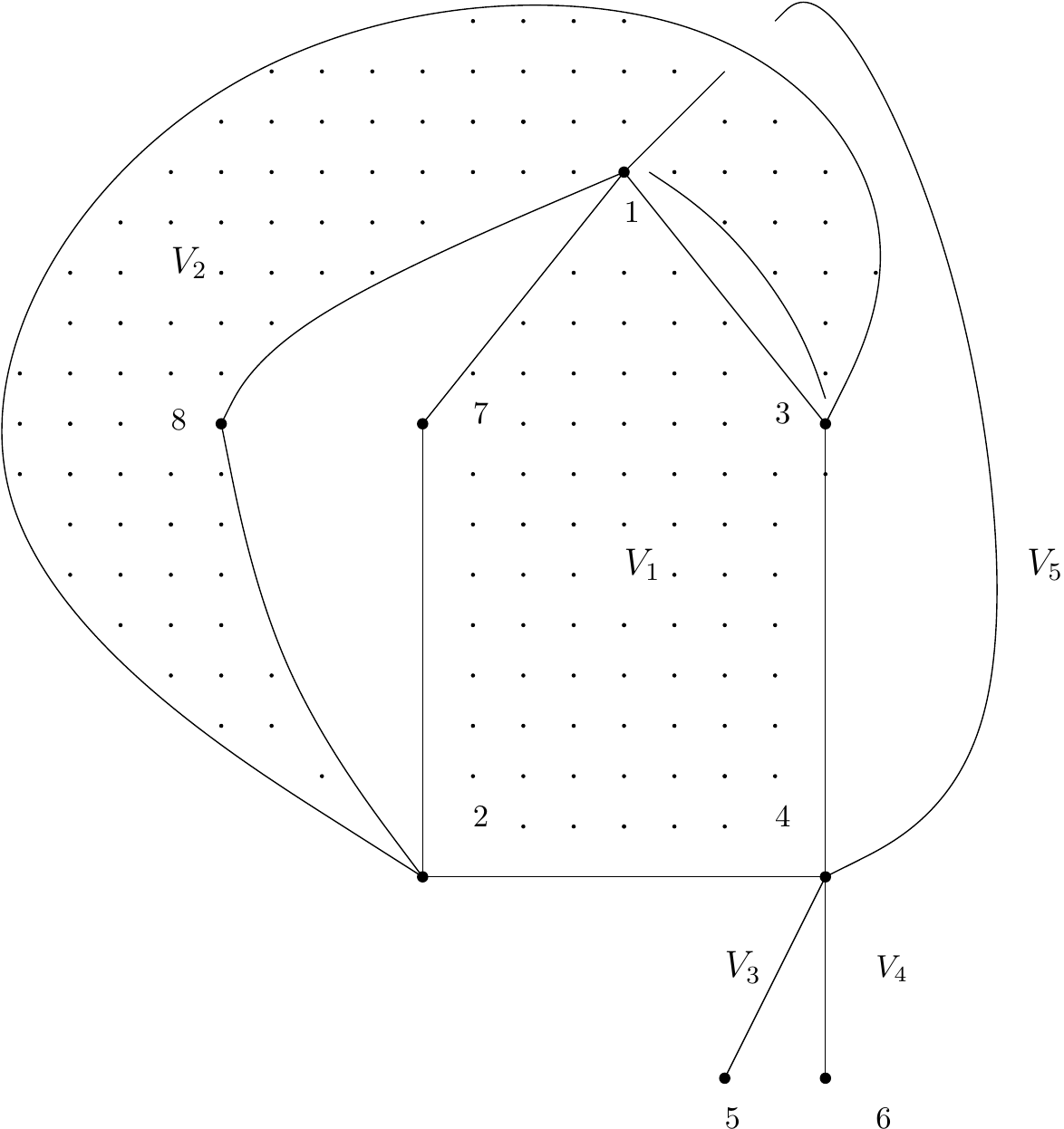} \end{center}

We note that these are two different realizations of the same \bc\ and that the first
realization is embeddable in the plane but the second is not.

Below is a realization of $\Delta$, the second example, in which
there are a number of self-foldings.

 \hskip 1in \begin{center} \includegraphics[scale=.5]{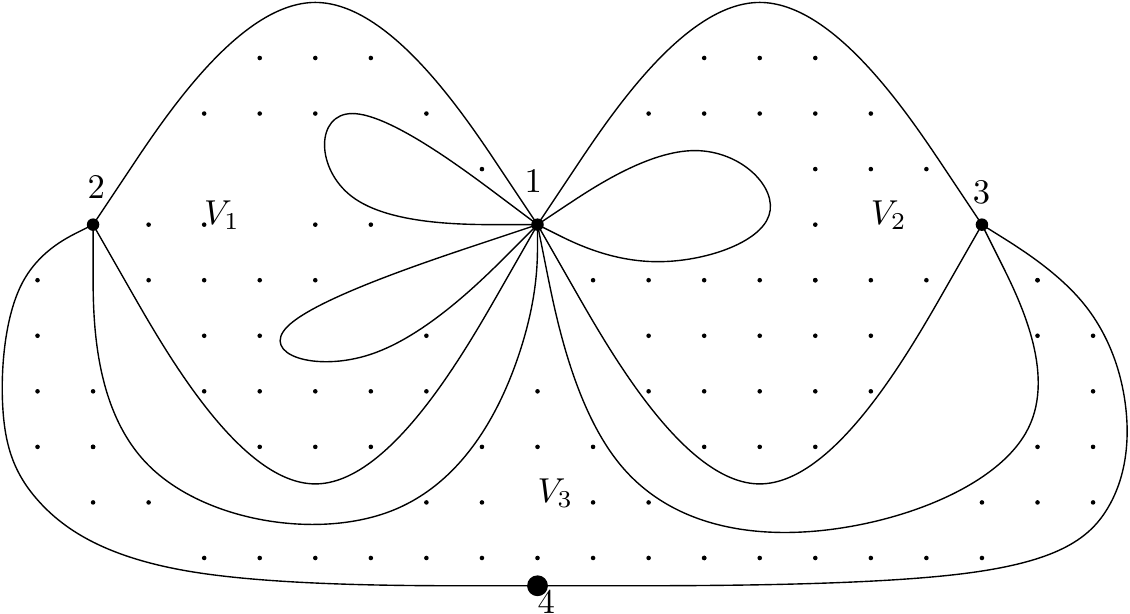}\end{center}

 }\end{Example}



\subsection{Truncated vertices and reduced Brauer configurations}\label{subsec-trunc and red}

We now define the crucial concept of a truncated vertex in
a Brauer configuration $\Gamma=(\Gamma_0,\Gamma_1,\mu, \mathfrak o)$.
 We say a vertex $\alpha\in\Gamma_0$ is \emph{truncated} if $\val(\alpha)\mu(\alpha)= 1$;
that is, $\alpha$ occurs exactly once in exactly one $V\in \Gamma_1$ and $\mu(\alpha)=1$.
A vertex
 that is not truncated is called a \emph{nontruncated vertex}.
\begin{Example}{\rm In the first example, vertices 6,7, and  8 are truncated.  Note that vertex 5 is
not truncated, even though $\val(5)=1$, since $\mu(5)=2>1$.  In the second example,
vertex 4 is truncated.}\end{Example}

We will see in what follows that a truncated vertex only plays a role if
it is one of the two vertices of a 2-gon.
Namely, we introduce a reduction procedure which removes truncated vertices from polygons with 3 or more vertices. This will not affect
the associated Brauer configuration algebra defined in Section \ref{sec-bca}, see Proposition \ref{prop-red}.

The reduction procedure for removing a truncated vertex occurring in a $d$-gon, $d\ge 3$ is defined as follows.
 Suppose that
 $\Gamma=(\Gamma_0, \Gamma_1,\mu,\mathfrak o)$ is a Brauer configuration and  that $\alpha\in \Gamma_0$ is a truncated vertex in a
 $d$-gon $V\in \Gamma_1$ with $d\ge 3$.  Note that since $\alpha$ is truncated,
$\val(\alpha)=1$ and this implies that $V$ is the unique polygon in $\Gamma_1$ having
$\alpha$ as a vertex.
 After reordering the
vertices in $V$, we may assume that $V$ is the $d$-gon
$(\alpha_1,\dots,\alpha_{d})$, with $\alpha=\alpha_{d}$.
Let $\Gamma'=(\Gamma'_0, \Gamma'_1,\mu', \mathfrak o')$, where
\begin{enumerate}
\item
$\Gamma'_0=\Gamma_0\setminus \{\alpha\}$,
\item $V'$ be the $(d-1)$-gon $\{\alpha_1,\dots,\alpha_{d-1}\}$,
\item  $\Gamma'_1= (\Gamma_1\setminus \{V\}) \cup \{V'\}$,
\item  $\mu'=\mu_{\mid \Gamma'_0}$.
\item $\mathfrak o'$ is the orientation induced from the orientation $\mathfrak o$.
\end{enumerate}
  We see that $\Gamma'$ is simply obtained
from $\Gamma$ by ``removing'' the truncated vertex $\alpha$ from $V$.  Note that
the number of polygons in $\Gamma$ and $\Gamma'$ are the same and
only one polygon in $\Gamma'$ has one less vertex.

If $\Gamma'$ also  has a truncated vertex in a $d'$-gon, with $d'\ge 3$, we can remove
it and obtain a Brauer configuration $\Gamma''$, where $\Gamma''$ has
2 less vertices than $\Gamma$.   Continuing in this fashion, we arrive at a
Brauer configuration $\Gamma^*$ in which, if $\alpha$ is a truncated vertex, it occurs in
a $2$-gon and hence there are no more reductions that can be performed in $\Gamma^*$.
We call $\Gamma^*$ a \emph{reduced Brauer configuration associated to $\Gamma$}.  If $\Gamma=\Gamma^*$ we
say that $\Gamma$ is a \emph{reduced} Brauer configuration.
We leave the proof of the next result to the reader.

\begin{lemma}\label{reduced}
Let $\Gamma$ be a Brauer configuration and  suppose that
$\Gamma^*$ and $\Gamma^{**}$  are two reduced configurations
associated to $\Gamma$.  Then we may choose a relabeling of the vertices
of\  $\Gamma^{**}$ so that $\Gamma^*=\Gamma^{**}$.
\qed
\end{lemma}

The above lemma allows us to talk about `the' reduced configuration associated
to a Brauer configuration.

\begin{Example}{\rm In the first example, the reduced configuration is obtained by removing
vertices 7 and 8.  Note that although vertex 6 is truncated, it is in the 2-gon
$V_4$.  In the second example,  the reduced configuration is obtained
by removing vertex 4.

Realizations of the reduced \bc s for these two examples are given below.
\begin{center}
\includegraphics[scale=.5]{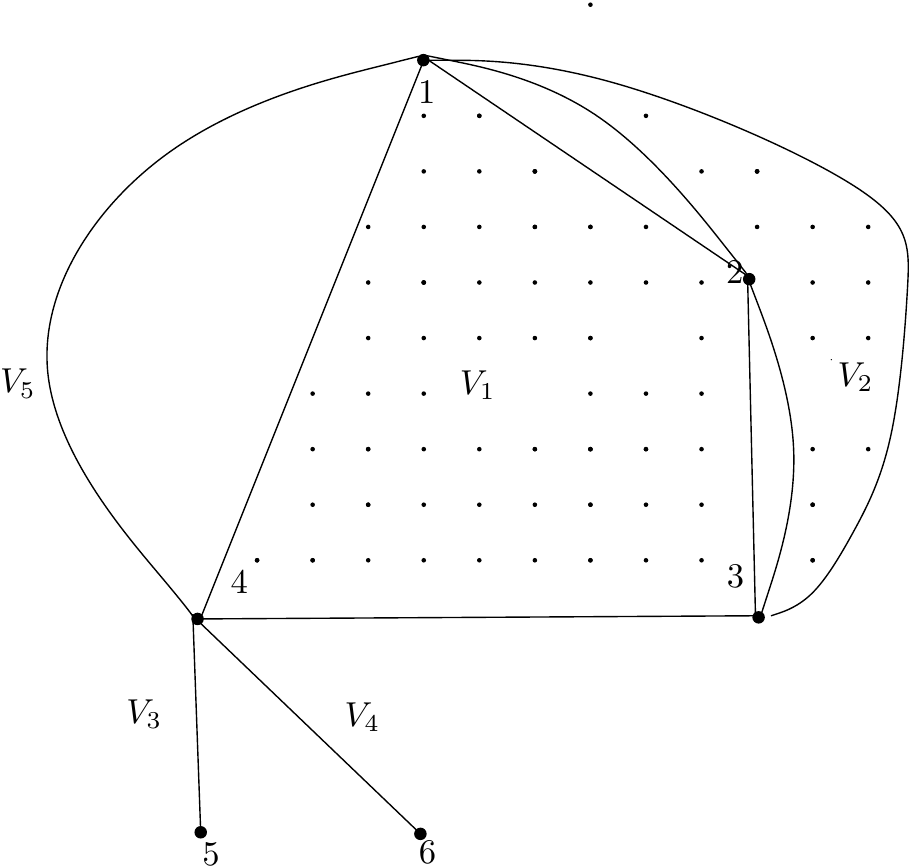}\quad\quad \quad\quad
 \includegraphics[scale=.6]{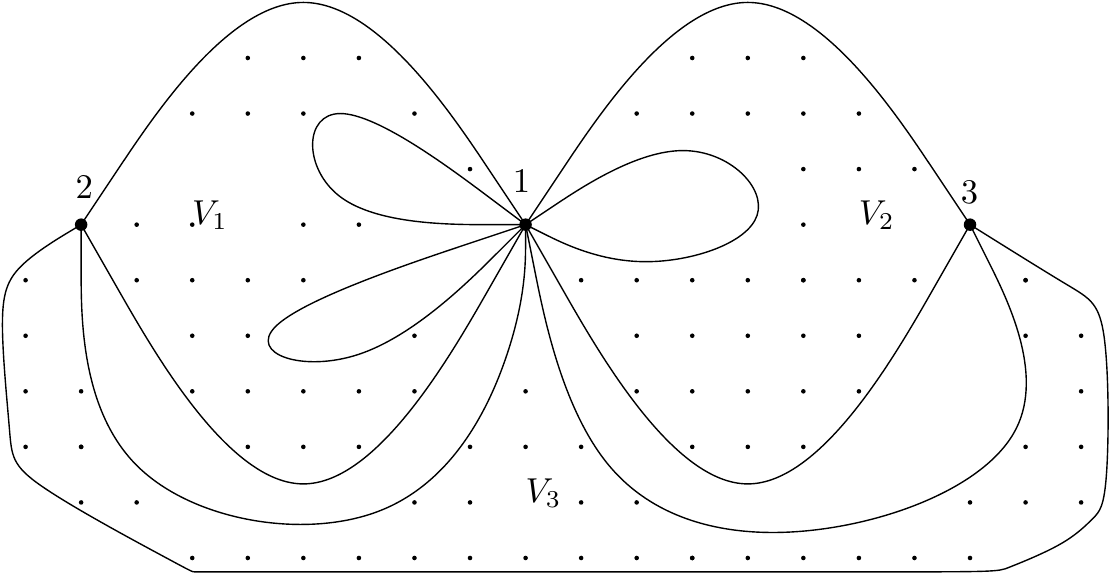}
\end{center}
}

\end{Example}

Given a \bc\  $\Gamma=(\Gamma_0, \Gamma_1,\mu, \mathfrak o)$, we
note that $\Gamma$ is
reduced if and only if every polygon $V$ in $\Gamma_1$ satisfies
one of
the following conditions:
\begin{enumerate}
\item $V$ contains no truncated vertices.
\item $V$ is a 2-gon with one truncated vertex.
\end{enumerate}




\section{Brauer configuration algebras}\label{sec-bca}

 In this section we  define Brauer configuration algebras. As described in the introduction, Brauer configuration algebras are generalizations of Brauer graph algebras.

Let $\Gamma=(\Gamma_0, \Gamma_1,\mu, \mathfrak o)$ be a Brauer configuration.
For each nontruncated vertex $\alpha\in\Gamma_0$, consider the list
of polygons  $V$ containing $\alpha$ such that $V$ occurs in this list $\oc(\alpha,V)$ times. As
described in Section \ref{section-config}, the orientation $\mathfrak o$ provides a cyclic ordering
of this list.
We call such a cyclically ordered list the \emph{successor sequence
at $\alpha$}.
 Suppose  that $V_1<\dots<V_t$ 
is the
successor sequence at some nontruncated vertex $\alpha$ (with $\val(\alpha)= t$).   Then we say that
$V_{i+1}$ is the \emph{successor of $V_i$ at $\alpha$}, for $1\le i\le t$, where
$V_{t+1}=V_1$.  Note that if $\val(\alpha)=1$, $\mu(\alpha)>1$ and $\alpha$ is vertex in
polygon $V$, then the successor sequence at $\alpha$ is
just $V$.

\begin{Example}{\rm The successor sequences for our two examples are already
given by the orientations in Example \ref{ex-order}.  For instance, for the first example (with orientation $\mathfrak{o}_1(\Gamma)$), the successor sequence
of vertex 4 is $V_1< V_4<V_3<V_5$ (or  $V_4<V_3<V_5<V_1$ etc.)  For the
second example, the successor sequence for vertex 1 is
$V_1^{(1)}< V_1^{(2)}<V_1^{(3)}<V_2^{(1)}<V_2^{(2)}<V_3.$}\end{Example}

 A Brauer configuration algebra $\Lambda_{\Gamma}$
associated to a Brauer configuration $\Gamma$ is defined by giving $\Lambda_{\Gamma}$
as a path algebra of a quiver modulo an ideal of relations. Fix a field $K$ and
let  $\Gamma=(\Gamma_0, \Gamma_1,\frak o,\mu)$ be a Brauer
configuration, with $\Gamma_1=\{V_1,\dots,V_m\}$.

\subsection{The quiver of a Brauer configuration algebra}

We define the quiver $\cQ_{\Gamma}$ as follows.   The vertex set $\{v_1,\dots, v_m\}$
of $\cQ_{\Gamma}$ is in correspondence with the set of polygons $\{V_1,\dots, V_m\}$ in $\Gamma_1$, noting that there is one vertex in $\cQ_\Gamma$
for every polygon in $\Gamma_1$.
We call $v_i$ (respectively, $V_i$) the vertex (resp. polygon) \emph{associated}
to $V_i$ (resp. $v_i$).  In order to define the arrows in $\cQ_{\Gamma}$, we
use the successor sequences.
For each nontruncated vertex $\alpha\in\Gamma_0$, and each successor $V'$ of $V$ at
$\alpha$, there is an arrow from
$v$ to $v'$ in $\cQ_{\Gamma}$,  where $v$ and $v'$ are the vertices in
$\cQ_{\Gamma}$ associated to the polygons $V$ and $V'$ in $\Gamma_1$, respectively.

Note that  $V'$ can be the successor of $V$ more than once at a given vertex of $\Gamma_0$, and also that
$V'$ can be the successor of $V$ at more than one vertex of $\Gamma_0$.  For each such
occurrence there is an arrow from $v$ to $v'$.  In particular, $\cQ_\Gamma$ may have  multiple
arrows from $v$ to $v'$.

Thus, every arrow in $\cQ_{\Gamma}$
is associated to a non-truncated vertex $\alpha\in\Gamma_0$ and two   polygons $V$ and $V'$ in $\Gamma_1$ such that $V'$ is the
successor of $V$ at $\alpha$.   Conversely, associated to two polygons
$V$, $V'$,  such that $V'$ is the
successor of $V$  at some vertex $\alpha\in\Gamma_0$,
there is an arrow from $v$ to $v'$ in $\cQ_{\Gamma}$.

\begin{Example}{\rm For the first example, recall that we took $V _1<V_5<V_2$ as
as the ordered list of polygons at vertex 1 for the orientation $\mathfrak o_1(\Gamma)$.
It is the successor sequence at vertex 1 of $\Gamma$ that
yields the arrows $a_1,a_2,$ and $a_3$ in the quiver
below.  For example, $V_5$ is the
successor of $V_1$ at vertex 1 yielding the arrow $a_1$.
The successor sequence $V_1<V_2$ at vertex 2  yields
the arrows $b_1$ and $b_2$.
 The successor sequence at vertex 3 yields
 $c_1$ and $c_2$, that at vertex
4  yields  $d_1,\dots, d_4$, and that at vertex 5  yields $e$.  }
\end{Example}

 The quiver associated to $(\Gamma_0,\Gamma_1,\mu,\mathfrak o_1(\Gamma))$ is

\hskip 1in  \begin{center} \includegraphics[scale=.5]{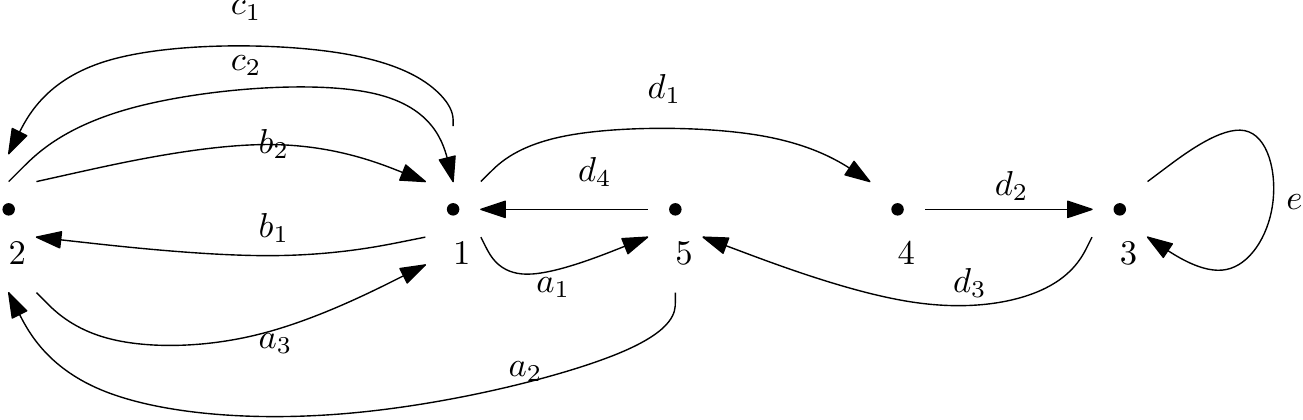} \end{center}

 The quiver associated to $(\Gamma_0,\Gamma_1,\mu,\mathfrak o_2(\Gamma))$ is

\hskip 1in \begin{center} \includegraphics[scale=.7]{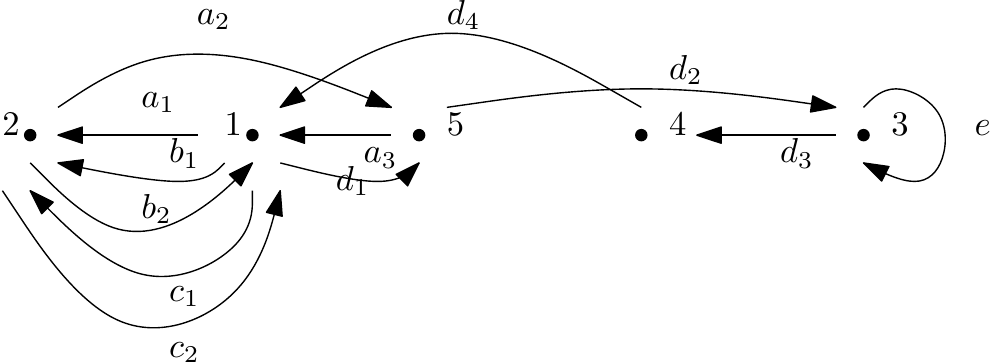} \end{center}

 We note that the two orientations $\mathfrak o_1(\Gamma)$ and $\mathfrak o_2(\Gamma)$ give rise to two non-isomorphic quivers; that is, they are non-isomorphic as oriented graphs. More precisely,  the number of arrows going into a vertex, the number of arrows going out of a vertex and the fact that vertex 3 has a loop show that if there were
an isomorphism of quivers, vertex 3 would be sent to vertex 3 and vertex 5 would be sent
to vertex 5.  But the arrow between vertices 3 and 5 are in opposite directions in
the two quivers. Hence the quivers are not isomorphic.

The quiver associated to $\Delta =
(\Delta_0,\Delta_1,\mu,\mathfrak o(\Delta))$ is

 \begin{center} \includegraphics[scale=.7]{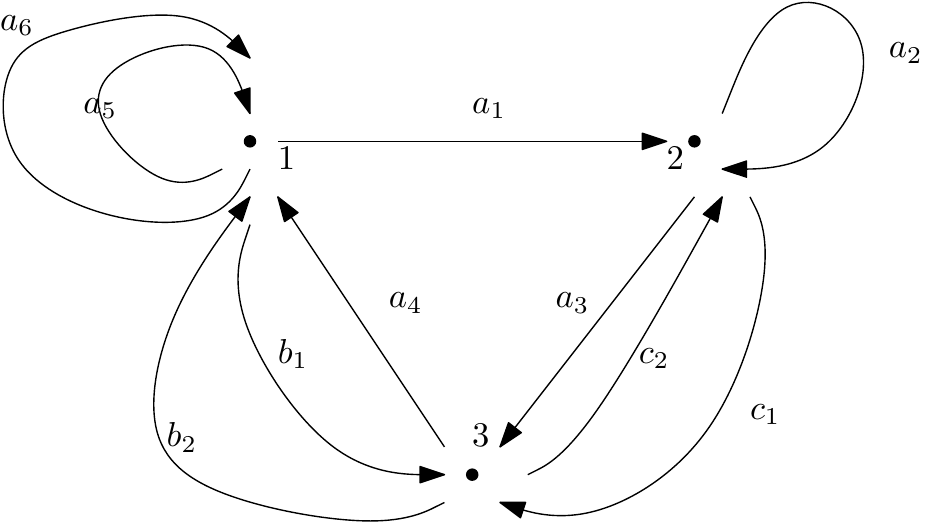}  \end{center}

\begin{remark}\rm When we apply the above construction to Brauer configuration algebras where all polygons are 2-gons,  we recover the usual quiver of a Brauer graph
algebra.
\end{remark}

\subsection{Ideal of relations and definition of a Brauer configuration algebra}

We define a set of elements $\rho_{\Gamma}$ in $K\cQ_{\Gamma}$ which will generate
the ideal of relations  $I_\Gamma$ of the Brauer configuration algebra associated to the Brauer configuration
$\Gamma$.  There are three types of relations in  $\rho_{\Gamma}$.  For this we need
the following definitions.

For each nontruncated vertex $\alpha\in\Gamma_0$ with successor sequence
$ V_1< V_2 <\ldots < V_{\val(\alpha)}$,
 let  $C_{j} = a_{j} a_{j+1}
\cdots a_{\val(\alpha)} a_1 \cdots a_{j-1}$ be the cycle in $\Qg$, for $j = 1, \dots, \val(\alpha)$,  where the arrow $a_r$ corresponds to
the polygon $V_{r+1}$ being the successor of the polygon $V_r$ at the
vertex $\alpha$.
 Now fix a polygon $V$ in $\Gamma_1$ and suppose that $\oc(\alpha, V) = t \ge 1$. Then there are $t$ indices $i_1, \dots, i_t$ such that $V = V_{i_j}$.
We define
the  \emph{special $\alpha$-cycles at $v$} to be  the cycles $C_{i_1}, \dots, C_{i_t}$,
where $v$ is the vertex in the quiver of $\cQ_{\Gamma}$ associated to  the polygon $V$.  Note that
each $C_{i_j}$ is a cycle in $\cQ_{\Gamma}$, beginning and ending
at the vertex $v$ and if  $\alpha$ occurs only  once
in $V$  and $\mu(\alpha) = 1$, then there is only one special $\alpha$-cycle at $v$. Furthermore, if  $V$ is a polygon
containing  $n$ vertices, counting repetitions, then there are a total of $n$
different special $\alpha$-cycles at $v$,   one for each $\alpha\in V$ where each repetition of $\alpha$ in $V$ gives rise to a different special $\alpha$-cycle at $v$. Note however that the special cycles, corresponding to repetitions of the same vertex $\alpha$ in $V$, are cyclic permutations of each other. 

We will sometimes say $C$ is a special $\alpha$-cycle if $v$ is understood
or just a special cycle if both $v$ and $\alpha$ are understood. 

\begin{Example}{\rm
In the first example with  orientation $\mathfrak o_1(\Gamma)$,  the special $1$-cycle
at $v_1$
is $a_1a_2a_3$, the special $2$-cycle at $v_1$ is $b_1b_2$, etc.  The 
special $1$-cycle at $v_2$ is the cyclic permutation
$a_3a_1a_2$ of the special $1$-cycle at $v_1$.  Similarly, the special $2$-cycle at $v_2$ is $b_2b_1$.  There are no
special $j$-cycles for $j=6,7,8$ since they are truncated vertices.  Note that  $e^2$ is the
unique $5$-cycle at $v_3$ since $\mu(5)=2$.

In the second example,  there are three special $1$-cycles at $v_1$,
 $a_1a_2\cdots a_6$, $a_6a_1a_2\cdots a_5$,
and  $a_5a_6a_1a_2\cdots a_4$.  There are two special $1$-cycles at $v_2$  $a_2a_3\cdots a_6a_1$ and
$a_3a_4\cdots a_6a_1a_2$.  There is only one special $1$-cycle at $v_3$.  Since vertex 4 of $\Gamma$
is truncated, there are no special $4$-cycles.  Since $2$ is not a vertex in $V_2$, there are no special
$2$-cycles at $v_2$ and similarly, since 3 is not in $V_1$ there are no special $3$-cycles at $v_1$.}
\end{Example}

 We now define the three types of relations by setting:

\textit{Relations of type one.}  For each polygon $V =\{\alpha_1, \ldots, \alpha_m\} \in \Gamma_1$ and each pair
of  nontruncated vertices $\alpha_i$ and $\alpha_j$  in $V$,  $\rho_\Gamma$ contains all relations of the form
 $C^{\mu(\alpha_i)} - (C')^{\mu(\alpha_j)}$ 
  where $C$ is a 
 special $\alpha_i$-cycle at $v$ and $C'$ is  a  special $\alpha_j$-cycle at $v$.

\textit{Relations of type two.} The type two relations are all paths of
the form $C^{\mu(\alpha)} a$ where
$C$ is a special  $\alpha$-cycle  and  $a$ is the first arrow in $C$.

\textit{Relations of type three.} These relations are
quadratic monomial relations of the form $ab$ in $K\cQ_{\Gamma}$ where $ab$
is not a subpath of any special cycle  unless $a = b$ and $a$ is a loop associated to a vertex $\alpha$ of valency 1 and $\mu(\alpha) >1$.

 \begin{definition}{\rm  Let $K$ be a field and
 $\Gamma$ a Brauer configuration.  The \emph{Brauer configuration
algebra $\Lambda_{\Gamma}$ associated to $\Gamma$}  is defined to be
 $K\cQ_{\Gamma}/I_\Gamma$, where  $\cQ_{\Gamma}$ is
the quiver associated to $\Gamma$ and $I_\Gamma$ is the ideal in  $K\cQ_{\Gamma}$ generated by the set of relations $\rho_{\Gamma}$
of type one, two and three.}
\end{definition}

 We note that the set of relations $\rho_\Gamma$ generating $I_\Gamma$ is not necessarily minimal and usually contains redundant relations.

\begin{Example}{\rm
For the first example, we list some of the relations of type one:
$a_1a_2a_3-b_1b_2, a_1a_2a_3 -(c_1c_2)^3, a_1a_2a_3-d_1d_2d_3d_4,
 a_3a_1a_2-b_2b_1,  a_3a_1a_2 -(c_2c_1)^3,  d_3d_4d_1d_2-e^2$, etc.   Many of the  type one relations are redundant, for example, $b_1b_2-(c_1c_2)^3$
and
$(c_1c_2)^3 -d_1d_2d_3d_4$  follow from the above.  Some of the type two relations are
\sloppy
$ a_1a_2a_3a_1, a_2a_3a_1a_2,  a_3a_1a_2a_3, d_4d_1d_2d_3d_4$, $d_3d_4d_1d_2 d_3$,
$e^3$,  $(c_1c_2)^3c_1$, $(c_2c_1)^3c_2$, etc.  The type three relations are any of the paths of length two of the form
 $a_ib_j$, $b_ja_i$, $a_ic_j$, $c_ja_i$, $a_id_j, d_ja_i,   ea_i,$ and $ a_ie$ for all possible combinations of $i,j$.
This gives a partial list of the type one, two, and three relations, and it includes many relations that are consequences of others.

For the second example, some of the type one  relations are   $a_1a_2\cdots a_6- a_6a_1a_2\cdots a_5, \\
a_1a_2\cdots a_6- a_5 a_6a_1a_2\cdots a_4,   a_2a_3\cdots a_6a_1- a_3a_4\cdots a_6a_1a_2$,
and $b_2b_1-c_2c_1$.
Some type two relations are $a_6a_1\cdots a_6,  a_1\cdots a_6a_1$, or
$b_2b_1b_2$. Any $a_ib_j, b_ja_i, a_ic_j, c_ja_i,
b_ic_j$ are type three relations.  Some  other relations of type three are
$a_6a_5, a_5^2, a_2^2$ and $ a_5a_1$.   } \end{Example}

 Next we show that the reduction procedure for removing truncated
vertices from a $d$-gon, $d\ge 3$, does not change the Brauer configuration.

\begin{prop}\label{prop-red} Let $\Gamma$ be a Brauer configuration with associated
Brauer configuration algebra $\Lambda_\Gamma$.   Suppose $\alpha\in \Gamma_0$
is a truncated vertex in a polygon $V\in \Gamma_1$ and $V$ is a $d$-gon,
$d\ge 3$.   Let $\Gamma'$ be the Brauer configuration algebra obtained
by removing the vertex $\alpha$ as in Section  \ref{subsec-trunc and red}. Then
the Brauer configuration algebra $\Lambda_{\Gamma'}$  associated to $\Gamma'$ is
isomorphic to $\Lambda_\Gamma$.

\end{prop}
\begin{proof}  Since $\alpha$ is truncated, there is a unique polygon, say $V\in\Gamma_1$
in which $\alpha$ is a vertex.  Again since $\alpha$ is truncated there are no
special $\alpha$-cycles.	 Thus, no arrows are created in the quiver of $\Lambda$ by $\alpha$,
and the quivers of $ \Lambda$ and $\Lambda'$ are the same.  Similarly, the ideals of relations
are seen to be the same and the result follows.
\end{proof}

\subsection{Special cycles.}\label{special cycles}
In this section we investigate properties of special cycles and
use these properties to obtain results about the muliplicative
structure of a \bca.

Let $\Lambda=K \cQ/I$ be the \bca\ associated to a
reduced \bc\ $\Gamma$. Denote by $\pi\colon K\cQ\to \Lambda$  the canonical
surjection. Then if no confusion can arise we denote $\pi(x)$ by $\bar x$, for $x \in \K \cQ$. 

We begin with a list of facts about successors and successor sequences translated
into facts about special cycles.  The proofs of these facts are
immediate consequences of the definitions and are left to
the reader.  We denote by (F) facts relating to successor sequences and by (F$'$) the analogous facts expressed in terms of special 
cycles.

\begin{enumerate}
\item[(F1)] If $V'$ is a successor to $V$ in the successor sequence at
the (nontruncated) vertex $\alpha\in\Gamma_0$ then there is a unique arrow in $\cQ$
from $v$ to $v'$ associated to $V'$ being the successor of $V$.
\item[(F$'$1)] If $a\colon v\to v'$ is an arrow in $\cQ$ then  up to cyclic permutation there is a unique 
special cycle $C$
in which  $a$ occurs.  If, in particular, $C$ is a special $\alpha$-cycle, for some $\alpha\in \Gamma_0$ then $a$ is associated to one occurence of
$V'$ being a successor of $V$ in the successor sequence at $\alpha$.
\item[(F2)] For each nontruncated vertex $\alpha\in\Gamma_0$ there
is a unique successor sequence at $\alpha$, up to cyclic permutation.
\item[(F$'$2)]  For each nontruncated vertex $\alpha\in \Gamma_0$, there
is a unique special $\alpha$-cycle, up to cyclic permutation.
\item[(F3)]  If $V'$ is the successor of $V$ at $\alpha\in \Gamma_0$,
 then, after cyclically reordering the successor sequence at $\alpha$,
the sequence begins with the chosen $V<V'$.
\item[(F$'$3)] If $a$ is an arrow in $\cQ$,  then there is a unique
nontruncated $\alpha\in \Gamma_0$ and a unique special $\alpha$-cyce $C$ such
that $a$ is the first arrow in $C$.  In particular, there are no
repeated arrows in a special cycle.
\item[(F$'$4)]  If there is an arrow that occurs in two special cycles
$C$ and $C'$, then there is a nontruncated vertex $\alpha\in\Gamma_0$ such
that both $C$ and $C'$ are special $\alpha$-cycles and $C'$ is
a cyclic permutation of $C$.  
\item[(F5)] The number of polygons in the successor sequence at
a vertex $\alpha\in\Gamma_0$ is $\sum_{V\in\Gamma_1}\oc(\alpha, V)$.
\item[(F$'$5)] The number of arrows in a special $\alpha$-cycle
is $\sum_{V\in\Gamma_1}\oc(\alpha, V)$.
\end{enumerate}

The next result and its corollary show that there is a very tight  multiplicative
structure in  $\Lambda$.

\begin{prop}\label{prop-mult} Let $\Gamma$ be a Brauer configuration with associated
Brauer configuration algebra $\Lambda = \K \cQ /I$ and let $V\in\Gamma_1$,  $\alpha\in\Gamma_0$  a
nontruncated vertex in $\Gamma_0$ that occurs
in $V$. Let $C=a_1a_2\dots a_{\val(\alpha)}$ be a  special $\alpha$-cycle at
$v$ and $C'$ the cyclic permutation $a_{j+1}\dots a_{\val(\alpha)}a_1\dots a_j$.  
Let $p=a_1a_2\dots a_j$, for some $1\le j\le \val(\alpha)-1$ and set
$x=C^sp$ and $y=p{C'}^s$, for some $ 0 \le  s<\mu(\alpha)$.
Then\begin{enumerate}
\item  $a_i\ne a_j$,  for $i\ne j$.
\item $\bar x\ne 0$.
\item If $a$ is an arrow in $\cQ$, then $\overline{xa}\ne 0$ if and only
if $a=a_{j+1}$.
\item $\bar y\ne 0$.
\item If $a$ is an arrow in $\cQ$, then $\overline{ay}\ne 0$ if and only
if $a=a_{\val(\alpha)}$.
\end{enumerate}
\end{prop}

\begin{proof}
Part (1) follows from (F$'$3).
The type two and three relations are monomial paths and the type one relations
are differences $D^{\mu(\alpha)}-{D'}^{\mu(\alpha)}$ where $D$
and $D'$ are special $\alpha$-cycles for some $\alpha\in\Gamma_0$.  Since $x$ and $y$ have no subpaths
that are type two relations or type three
relations, we see (2) and (4) hold.   Similarly, if $a\ne a_{j+1}$, then
$a_ja$ is a type three relation and hence (3) holds.  Finally,
 if $a\ne a_{\val(\alpha)}$ then
  $a a_1$ is  a type three relation and hence (5) holds.
\end{proof}

 Proposition~\ref{prop-mult} has the following
consequence which plays an important role in \cite{GreenSchroll}.

\begin{prop}\label{prop-cycle}   
Let $\Gamma$ be an indecomposable reduced Brauer configuration with associated
Brauer configuration algebra $\Lambda = \K \cQ /I$ and assume
$\rad^2(\Lambda)\ne 0$.
Let $a$ be an arrow in $\cQ$. Then
\begin{enumerate}
\item there is a unique arrow $b$ such that $\overline{ab}\ne 0$, and
\item there is a unique arrow $c$ such that $\overline{ca}\ne 0$.
\end{enumerate}
\end{prop}

\begin{proof}  First note that if $P$ is an indecomposable projective
$\Lambda$-module with $P\rad^2(\Lambda)=0$,  then by indecomposability and symmetry   and the definition of the relations,
$P/P\rad(\Lambda)$ and $P\rad(\Lambda)$ are isomorphic simple
$\Lambda$-modules, and hence $\Lambda$ would be isomorphic
to $K[x]/(x)^2$, contradicting the assumption that $\rad^2(\Lambda)
\ne 0$.

From the definition of the type one, two, and three relations, if $x$ and $y$
are arrows in $\cQ$ then $\overline{x y}\ne 0$ if and only
if there is a special cycle in which the arrows $x$ and $y$ occur and
$y$ immediately follows $x$.  By (F$'$3), suppose that $a$ is the
first arrow in the special $\alpha$-cycle $C=a_1\cdots a_{\val(\alpha)}$ for some nontruncated vertex
 $\alpha\in\Gamma_0$.   
It follows that $b=a_2$ and $c=a_{\val(\alpha)}$ and we are done.
\end{proof}

From this result we obtain the following surprising consequence
 which shows that
there is a tight connection between arrows in $\cQ$ and paths
of length 2 in $\cQ$ that are not in the ideal $I$, the ideal generated
by the relations of types one, two, and three.  We introduce the
following notation for this result.   Set
$$\Pi = \{   ab \; | \; a, b \mbox{ arrows in } \cQ \mbox{ and } ab \notin I \}$$
that is $\Pi$ is the set of paths of length 2 in $\cQ$ that are not in $I$.    
The following corollary follows
directly from \ref{prop-cycle}.

\begin{cor}\label{cor-path2}
Let $\Gamma$ be an indecomposable reduced Brauer configuration with associated
Brauer configuration algebra $\Lambda = \K \cQ /I$ and assume that
$\rad^2(\Lambda)\ne 0$.  Let $\Pi$ be as defined above and 
define $f\colon \Pi\to \cQ_1$ by $f(ab)=a$ and $g\colon \Pi\to \cQ_1$
by $g(ab)= b$, where $a,b\in\cQ_1$ with $ab\notin I$. Then the maps
$f$ and $g$ are bijections.
\end{cor}

 We call two special cycles in $A$ \emph{equivalent} if
one is a cyclic permutation of the other.  Suppose that there are 
$t$ equivalence classes and let $\cC = \{C_1, \dots, C_t\}$ be a full set of
equivalence class representatives.   

\begin{prop}\label{prop-complete}Let $\Gamma$ be an indecomposable reduced Brauer configuration with associated
Brauer configuration algebra $\Lambda = \K \cQ /I$ and assume that
$\rad^2(\Lambda)\ne 0$.  Let $\cC=
\{C_1, \dots, C_t\}$ be a complete set of representatives of special cycles.  The following statements hold.
\begin{enumerate}
\item  Any arrow of $\cQ$ occurs once in exactly one of the special cycles 
in $\cC$.
\item The cardinality of $\cQ_1$ is
 \[ \left| \cQ_1\right| = 
\sum_{C_i \in \cC } |C_i|  = \sum_{\stackrel{\alpha \in \Gamma_0,}{ \alpha \text{ nontruncated }}}
\sum_{V\in\Gamma_1}\oc(\alpha,V)
\] where $|C_i|$ denotes 
the number of arrows in the cycle $C_i$.
\end{enumerate}
\end{prop}

\begin{proof}
Given an arrow $a\in\cQ_1$,  there is a unique nontruncated
$\alpha\in\Gamma_0$ such that $a$ is in a special $\alpha$-cycle
by (F$'$1).  Hence $a$ is in every special $\alpha$-cycle since any
two special $\alpha$-cycles are cyclic permutations of one another
by (F$'$3) and (F$'$4).   Thus, we can assume that $a$ is in exactly one of the special cycles in $\cC$. 
Since special cycles 
have no repeated arrows by (F$'$3), part (1) follows.  Part (2) follows
from part (1) and we are done.
\end{proof}

\section{Properties of Brauer configuration algebras}\label{sec-basic}



 In this section we prove some basic properties of Brauer configuration algebras. Assume that
$\Lambda=K\cQ/I$ where $\Lambda$ is a \bca\ associated to a \bc\ and $\cQ$ is the
quiver of $\Lambda$. We assume from now on that all \bc s are reduced. We show that $I$ is an admissible ideal and that $\Lambda$ is
a symmetric algebra.
We also show that
 a Brauer
 configuration algebra is indecomposable if and only if its Brauer configuration is connected.  Finally, we show
that  $\Lambda$ is a multiserial algebra; that is we show that the heart of $\Lambda$
is a direct sum of uniserial modules.

\subsection{First properties of Brauer configuration algebras and a basis of  $\Lambda$}

If $p$ is a path in a quiver $\cQ$, we let $\ell(p) $ denote the length of $p$.

\begin{lemma}\label{lem-path}Suppose that  $\Lambda=K\cQ/I$ is the Brauer
configuration algebra associated to a Brauer configuration $\Gamma$.
  Let  $C$ be a special cycle and  let $p$ be a path of length $\ge 1$ in $\cQ$ such that
the first arrow of $p$ is the first arrow in $C$.
Then
 $\bar p\ne 0$ if and only if $p$ is a prefix of $C^{\mu(\alpha)}$.
\end{lemma}

\begin{proof}
Since the first arrow of $p=a_1\cdots a_m$ is the same as the first arrow of $C$, either $p$
is a prefix of $C^s$ for some $s$ or there is an $i$ such that $a_i$ is an arrow in $C $
but  $a_ia_{i+1}$ is not in $C$.
First assume $a_ia_{i+1}$ is not in $C$.  Then $a_ia_{i+1}$ is not in any special cycle
by Proposition \ref{prop-mult}(3).  Hence $a_ia_{i+1}$ is a type three relation and $\bar p=0$.

Assume that $C$ is a special $\alpha$-cycle for some nontruncated
vertex $\alpha\in\Gamma_0$.
Now suppose that $p$ is a prefix of $C^s$. Then either
$\ell(p)\le\ell(C^{\mu(\alpha)})$ or
$\ell(p)>\ell(C^{\mu(\alpha)})$.

First assume that $\ell(p)>\ell(C^{\mu(\alpha)})$.
Then $p$ contains $C^{\mu(\alpha)}a_1$ which is a type two relation.  Hence, $\bar p=0$.

Now suppose that $\ell(p)\le\ell(C^{\mu(\alpha)})$. Then $p$ contains no relations
of type two or three.  By the length assumption, type one relations do not
affect $p$ and we see that $\bar p\ne 0$. The proof is complete.
\end{proof}

 We let $J$ denote the two sided
ideal in $K\cQ$ generated by the arrows in  $\cQ$.
Recall that the ideal $I$ in $K\cQ$ is
\emph{admissible} if $J^N\subseteq I\subseteq J^2$, for some $N\ge 2$.
Clearly, if $I$ is admissible then $\Lambda$ is finite dimensional.

\begin{prop} \label{prop-fd} Let  $\Lambda=K\cQ/I$ be the Brauer
configuration algebra associated to a Brauer configuration $\Gamma$. Then
$I$ is admissible and $\Lambda$ is a symmetric algebra.
\end{prop}

\begin{proof}  From the definition of the three types of relations, we
see that $I$ is contained in $J^2$.  Consider the set 
\[\cS=\{C^{\mu(\alpha)}\mid \alpha \text{ is a nontruncated
vertex and } C \text{ is a  special }\alpha\text{-cycle } \}.\] Let $N=\max_{C^{\mu(\alpha)}\in\cS}(\ell(C^{\mu(\alpha)})) +1$.
If $p$ is a path of length $N$, then $p$ cannot be a prefix of
any element in $\cS$.  Note that by (F$'$3), 
we see that every arrow in $\cQ$ is the prefix of some cycle in $\cS$.
Using this observation and 
 Lemma \ref{lem-path}, we see that
$\bar p=0$; that is, $p\in I$.  Thus $J^N\subseteq I$ and
it follows that  $  I$   is admissible.

Using Lemma \ref{lem-path}, the reader may check
that the two sided socle of $\Lambda$ is generated by the
elements of $\cS$.  In fact, for each $V\in\Gamma_1$,  choose
a nontruncated vertex $\alpha$ of $V$ and  one  special $\alpha$-cycle at $v$,
$C_{V}$, in $\cQ$.  Then $\{\overline{C_{V}^{\mu(\alpha)}}\mid V\in\Gamma_1\}$
forms a $K$-basis of the two sided socle of $\Lambda$.

To show that $\Lambda$ is symmetric, let $\phi\colon \Lambda\to K$
be the $K$-linear form defined as follows.
Let $p$ be a path in $\cQ$.  Then $\phi(\bar p)=1$ if and only if
$p\in \cS$.  If $p\notin \cS$, let $\phi(\bar p)=0$.
It is easy to show that $\phi(ab)=\phi(ba)$.   That $\Ker \phi$ contains
no left or right ideals follows from the description of a $K$-basis
of the two sided socle of $\Lambda$ and that $\phi$ is $1$ on elements of $\cS$.
It follows that $\Lambda$ is a symmetric algebra.
\end{proof}

The next result provides a useful $K$-basis
of $\Lambda$.

\begin{prop}\label{prop-basis}Let $\Lambda$ be the Brauer configuration
algebra associated to the Brauer configuration $\Gamma$.   For each $V\in\Gamma_1$,  choose
a nontruncated vertex $\alpha$ of $V$ and  exactly one special $\alpha$-cycle
$C_{V}$ at $v$. Then keeping the above notation, we have  
\[\{\bar p\mid p \text{ is a proper prefix of some }C^{\mu(\alpha)} \text{ where }
C\text{ is a special }\alpha\text{-cycle}\}\cup \{\overline{C_V^{\mu(\alpha)}}\mid
V\in\Gamma_1\}\] 
is a $K$-basis of $\Lambda$. 
\end{prop}

\begin{proof} We have seen that  $\{\overline{C_V^{\mu(\alpha)}}\mid
V\in\Gamma_1\}$ is a $K$-basis of the socle of $\Lambda$.
Using that every arrow is the start of a special cycle, 
Lemma \ref{lem-path} and that the only relations affecting
proper subpaths of the special cycles are monomial relations
(types two and three), the result follows.
\end{proof}

\subsection{Decomposable and indecomposable Brauer configuration algebras}

We start by investigating disconnectedness of Brauer configurations.

\begin{prop}\label{prop-dis} Suppose the Brauer configuration
$\Gamma$ is disconnected and decomposes into Brauer configurations $\Gamma'\cup\Gamma''$.
Then the associated Brauer configuration algebra $\Lambda_{\Gamma}$
is isomorphic to the product $\Lambda_{\Gamma'}\times \Lambda_{\Gamma''}$.
\end{prop}

\begin{proof}
By definition of a disconnected Brauer configuration given in
Section \ref{section-config}, there can be no arrows between vertices in $\cQ_{\Gamma'}$
and $\cQ_{\Gamma''}$ and the result follows.

\end{proof}

We prove the converse.

\begin{prop}\label{conv-conn}
If the Brauer configuration $\Gamma$ is connected then
the Brauer configuration algebra associated to $\Gamma$
is indecomposable as an algebra.
\end{prop}

\begin{proof}
We show that if a Brauer configuration algebra is decomposable then the Brauer configuration is disconnected.
Suppose that $\Gamma$ is a Brauer configuration
with associated Brauer configuration algebra $\Lambda$.
Assume that
 $\Lambda$ is not indecomposable and that
$\Lambda\cong \Lambda'\times \Lambda''$.  Let $\cQ, \cQ'$, and $\cQ''$ be the
quivers of  $\Lambda, \Lambda'$ and $\Lambda''$ respectively.  Then
$\cQ$ is the disjoint union of $\cQ'$ and $\cQ''$.   Let \sloppy
$A=\{V \in \Gamma_1\mid \text{ the vertex in }\cQ\text{ associated to }V\text{ is in }\cQ'\}$ and
$B=\{V \in \Gamma_1\mid \text{ the vertex in }\cQ\text{ associated to }V\text{ is in }\cQ''\}$.
Note that  $ A\cup B=\Gamma_1$ and $A\cap B=\emptyset$.

Let $\cA$ be the set of vertices of the polygons in $A$ and $\cB$ be the set vertices of the polygons
in $B$.  Then we show that $\cA\cap\cB=\emptyset$.
Indeed, suppose there is a vertex $\alpha\in\Gamma_0$ that 	is a vertex
of a polygon $V\in A$ and of a polygon $V'\in B$.  Then both $V$ and $V'$ occur in the successor sequence at $\alpha$.
Hence, if $C$  is a special $\alpha$-cycle, both $v$ and $v'$ occur as vertices in $C$.  This contradicts the fact
that there are no paths from $v$ to $v'$ in the quiver of $\Lambda$ since
$v$   is a vertex in $\cQ'$ and $v'$ is a vertex in $\cQ''$.
  Hence,
by condition C1, we get a partition of $\Gamma_0 =\cA \cup \cB$.

 In order to show  that $\Gamma$ is disconnected, assume for a contradiction
 that $\Gamma$ is connected.  Hence, since $\Gamma_1 = A \cup B$, $A \cap B = \emptyset$, $\Gamma_0 =   \cA \cup \cB$  and $\cA\cap\cB=\emptyset$, for $\Gamma$ to be connected there must be some polygon $V$ that has
vertices from both $\cA$ and $\cB$; that is, there is a polygon $V\in\Gamma_1$ and
$\alpha\in \cA$, and $\beta\in\cB$, such that $\alpha $ and $\beta$ are vertices of $V$.
This contradicts $A\cap B=\emptyset$, finishing the proof.
\end{proof}

\subsection{Gradings of Brauer configuration algebras}


\begin{prop}\label{prop-grad} Let $\Lambda$ be the Brauer configuration
algebra associated to a connected  Brauer configuration $\Gamma$.  The algebra
 $\Lambda$ has a length grading induced from the path algebra $K\cQ$
if and only if there is an $N\in\mathbb Z_{>0}$ such that, for each nontruncated vertex
$\alpha$ in $\Gamma_0$, $\val(\alpha)\mu(\alpha)=N$. 
\end{prop}

\begin{proof}  If $\alpha\in\Gamma_0$ is a truncated vertex in the
2-gon $V$, then the projective indecomposable associated to $V$ is
uniserial and gives rise to a monomial relation of type two.  Monomial
relations are homogeneous under any grading.

 Suppose $\Lambda$ has a length grading.
If $C$ is a special $\alpha$-cycle, then the length of $C^{\mu(\alpha)}$ is $\val(\alpha)\mu(\alpha)$.
Since type two and three relations are  monomials,   the ideal of relations $I$ is generated
by length homogeneous relations if the type one relations are
length homogeneous.  Thus, for $I$ to be generated by length homogeneous elements, all type one relations must be length homogeneous. But this
implies that
 for a  special $\alpha$-cycle $C$ at $v$ and a special $\beta$-cycle $C'$ at $v$, the relation   $C^{\mu(\alpha)}-(C')^{\mu(\beta)}$  must be length homogeneous.
Thus for all vertices $\alpha$
in a polygon, the $\val(\alpha)\mu(\alpha)$ must all be equal.  Using connectedness, we see
that if $\alpha$ is a vertex in  polygon $V$ and $\beta$ is a vertex in polygon $V'$,  then there
is a sequence of polygons $V_1,\dots, V_k$ and vertices $\alpha_i,\beta_i$ in $V_i$, such
that $\alpha=\alpha_1$, $\beta_i=\alpha_{i+1}$ for $i=1,\dots,k-1$ and
$\beta=\beta_k$. The result now follows.

The converse immediately follows from the fact the if  for each nontruncated vertex
$\alpha$ in $\Gamma_0$, $\val(\alpha)\mu(\alpha)=N$ then all relations are length homogeneous.
\end{proof}

\subsection{Projective indecomposable modules and uniserial modules}\label{subsec-mult}
We now describe the projective-injective indecomposable modules and the non-projective uniserial modules over a Brauer configuration algebra.

Let $\Lambda$ be a \bca\ associated to a reduced Brauer configuration $\Gamma$.
In what follows we adopt the following notation, if $V$ is a polygon in $\Gamma$ and $v$ is the vertex in
the quiver of $\Lambda$ associated to $V$, then we let $P_V$ be the projective $\Lambda$-module (resp. $S_V$ the simple
$\Lambda$-module)
associated to $v$.
Let $V$  be a  polygon in $\Gamma_1$, $\alpha$ a nontruncated vertex in $V$ with $n = \val(\alpha)$, and $C=a_1a_2\cdots a_{n}$
 a special $\alpha$-cycle at $v$.  Let $a_i$ be an arrow from $v_i\to v_{i+1}$ and let $V_i$
be the polygon in $\Gamma_1$ associated to $v_i$, for $i=1,\dots,{n}$.  Note that $V_1=V_{{n}+1}=V$.
It follows from  Proposition \ref{prop-cycle}  and Proposition \ref{prop-basis} that based on the choice of $C$, we can now
define a chain of uniserial submodules in the following way.

Set $U_{n\mu(\alpha)} = U_{n\mu(\alpha)}(C)$ to be isomorphic to the simple module
associated to $v=v_{{n}+1}$. Note that $U_{n\mu(\alpha)}$ has $K$-basis $ \{ \overline{(a_1 a_2 \dots a_{n})^{\mu(\alpha)}} \} =\{\overline{
C^{\mu(\alpha) }} \}$. Assuming
$U_{j+1} = U_{j+1}(C)$ is defined and $j\ge 1$,
let $U_{j} = U_j(C)$ be the uniserial $\Lambda$-module
containing $U_{j+1}$ and such that $U_{j}/U_{j+1}$  is isomorphic to the simple module
associated to $v_{\ell}$, where $j=kn +\ell$ with $1\le \ell< n$ and $0 \leq k < \mu(\alpha)$.
  Note that $U _ {j}$ has $K$-basis
$\{(\overline{C^ka_1a_2\dots a_{\ell}}), (\overline{C^ka_1a_2\dots a_{\ell+1})},\dots,(\overline{
C^{\mu(\alpha) }})\}$.
Thus, we obtain a chain of uniserial  modules
\[ (0)\subset U_{{n\mu(\alpha)}}\subset U_{(n\mu(\alpha))-1}\subset \cdots \subset U_{2} \subset U_{1}. \]

If $V$ is a $2$-gon with vertices $\alpha$ and $\beta$, with $\beta$ truncated (and hence $\alpha$
is nontruncated),  then the indecomposable projective $\Lambda$-module $P_V$ is uniserial \sloppy with $K$-basis
$\{e_v, \overline{a_1},\overline{a_1a_2},\dots,\overline{(a_1a_2\dots a_{{n}})^{\mu(\alpha)}} \}$, where $e_v$ is the primitive idempotent in $\Lambda$ at vertex $v$,
 and we obtain a chain of uniserial  modules
\[ (0)\subset U_{{n\mu(\alpha)}}\subset U_{{(n\mu(\alpha))}-1}\subset \cdots \subset U_{2}\subset U_{1}\subset P_V. \]

\begin{lemma}\label{prop-unis} Let $\Lambda$ be a Brauer configuration algebra with connected Brauer configuration $\Gamma$.
Let $U$ be a non-projective uniserial $\Lambda$-module.  With the notation above we have the following.
\begin{enumerate}
\item
There is a
polygon $V$ in  $\Gamma$, a  nontruncated vertex $\alpha$ in $V$,
and a special $\alpha$-cycle $C$ at $v$ such that $U$ is isomorphic
to $U_j(C)$ for some $1 \le j \le \val(\alpha)\mu(\alpha)$.   Furthermore, if $U$ is not a simple $\Lambda$-module
then $V,\alpha,$ and $j,$  are unique and $C$ is unique up to cyclic permutation.

\item
Suppose that $U$ is a non-zero, and non-simple uniserial module isomorphic to
$U_j(C)$  obtained from a special $\alpha$-cycle $C$
 and let $U'$ be a non-projective uniserial $\Lambda$-module isomorphic to
$U'_{j'}(C')$ obtained from a special $\beta$-cycle $C'$ at $v'$, with $\alpha\ne \beta$.
 Then
$\dim_{K}\Hom_{\Lambda}(U,U')\le 1$.
\end{enumerate}
\end{lemma}

\begin{proof}
Let $U$ be a uniserial $\Lambda$-module.  If $U$ is a simple
$\Lambda$-module then (1) holds.  Assume that $U$
is a nonsimple uniserial $\Lambda$-module.  Let the socle of $U$
be isomorphic to a simple module $S_V$ for some $V\in\Gamma_1	$.
Since $\Gamma$ is reduced, $V$ is either a $2$-gon with one truncated vertex
or $V$ is a $d$-gon, $d\ge 2$, and all the vertices of $V$ are nontruncated.
We have that $U$ maps monomorphically into $P_V$ since $P_V$ is an
indecomposable injective module.  Applying Proposition \ref{prop-mult} and
Propositiom \ref{prop-cycle}, we see that $U$ is isomorphic to $U_j(C)$, for
some special $\alpha$-cycle at $v$ and some $j$, $1 \le j\le \val(\alpha)\mu(\alpha)$.  The last
part of (1) again follows from  Proposition \ref{prop-mult} and
Proposition \ref{prop-cycle}

Part (2) follows from  (F$'$4) since, if $\alpha\ne \beta$ then $C$ and $C'$ can
have no arrows in common.  In particular, the only possible
map would be from $U$ to the socle of $U'$.
\end{proof}

\begin{prop}  Let $\Lambda$ be a Brauer configuration algebra associated
to a connected  Brauer configuraton $\Gamma$.
Let $U$ be a uniserial $\Lambda$-module.
With the notation above we have that
$U$ is projective (uniserial) if and only if
$U$ is isomorphic to the indecomposable projective
$\Lambda$-module associated to a $2$-gon $V$ having a
truncated vertex.
\end{prop}

\begin{proof}
If $V$ is a $2$-gon with vertices $\alpha$ and $\beta$ and $\beta$
is truncated, then there is only one special $\alpha$-cycle at $v$ and by
the discussion at the beginning of this section, $P_V$ is a uniserial
projective module.  Conversely, suppose that  $V$ is a $d$-gon, $d\ge 2$, such that
each vertex in $V$ is nontruncated.  Let $V=\{\alpha_1,\dots, \alpha_d\}$ and, for
$1\le i\le d$, let $C_i$ be a special $\alpha_i$-cycle at $V$. These cycles yield distinct
uniserial submodules $U_1(C_i)$, all having the same
socle. Hence $P_V$ is
not a uniserial module.  The result follows.
\end{proof}

\begin{Example}{\rm
Let $\Lambda$ be the \bca\ associated to the \bc\ of our first example (after reducing). If $X$ is a set of
elements in $\Lambda$, let $\langle X\rangle$ denote the right submodule of  $\Lambda$ generated by $X$. Then, the  polygon
$V_1=\{1,2,3, 4\}$, yields the following sequences of uniserial submodules:
\begin{gather*}\langle \overline{a_1a_2a_3}\rangle\subset \langle \overline{a_1a_2}\rangle\subset
\langle \overline{a_1}\rangle, \\
 \langle \overline{b_1b_2}\rangle\subset \langle \overline{b_1}\rangle,\\
\langle (\overline{c_1c_2})^3\rangle\subset \langle (\overline{c_1c_2})^2\overline{c_1}\rangle\subset
\cdots\subset
\langle \overline{c_1}\rangle, \\
\langle \overline{d_1d_2d_3d_4}\rangle\subset \langle \overline{d_1d_2d_3}\rangle\subset
\cdots \subset
\langle \overline{d_1}\rangle.
\end{gather*} Note that vertex 5 in $V_3$ is truncated and hence we get a uniserial projective
$\Lambda$-module generated by the idempotent $e_{v_3}$.  The sequence of uniserial modules
in $e_{v_3}\Lambda$ is
\begin{gather*}\langle \overline{d_2d_3d_4d_1}\rangle\subset \langle \overline{d_2d_3d_4}\rangle\subset
\cdots
\langle \overline{d_2}\rangle \subset \langle e_{v_3}\rangle.
\end{gather*}Since the \bc\ of our second example (after reducing) has no truncated vertices,
the associated \bca\ has no  uniserial projective modules.
}\end{Example}

Next, we describe the indecomposable projective modules over a Brauer configuration algebra.

\begin{thm}\label{prop-proj}
 Let $\Lambda$ be  a  \bca\ associated to a reduced \bc\ $\Gamma$.
Let $P$ be an indecomposable projective $\Lambda$-module associated to
a $d$-gon $V$.  Define an integer $r$ by setting $r=d$ if all vertices of  $V$ are nontruncated and $r=1$
if  $V$ is a  $2$-gon with one truncated vertex.  Then $rad(P)$ is the sum of the $r$ uniserial $\Lambda$-modules
$\sum_C U_1(C)$ where $C$  runs over the  special 
$\alpha$-cycles at $v$
for every nontruncated vertex  $\alpha$ of $V$.  If $r>1$, then $U_i\cap U_j$ is the simple socle of $P$.    Furthermore, the heart of $P$,  $\rad(P)/ \soc(P)$ is a direct sum of uniserial
$\Lambda$-modules.

\end{thm}

\begin{proof} Let \[\cC_v=\{C\mid C\text{ is a special }\alpha\text{-cycle at }v,  \alpha \text{ a
nontruncated vertex in } V \}=\{C_1,\dots, C_r\}.\]
 For each $C_i\in\cC_v$, the uniserial $\Lambda$-module $U_1(C_i)$ is  generated by the first arrow, $a_i$, in $C _i$.
 Let $M$ denote the submodule of $P$ generated by
$a_1,\dots, a_{r}$.  Note that $a_1,\dots, a_{r}$ are precisely the arrows in
the quiver of $\Lambda$ that start at  vertex $v$ (where $v$ corresponds to $V$) and that
 $P = e_v \Lambda$, where $e_v$ is the primitive idempotent in $\Lambda$ at $v$.
  Then it follows that $M=\rad(P)$ and that $M=\sum_{i=1}^rU_1(C_i)$.

For $i=1,\dots r$, we have seen that $\ov{C_i^{\mu(\alpha)}}$ is a nonzero
element in the socle of  $P$ and, considering the relations of type 1
they are all equal.  Since the type 2 and type 3 relations are monomial relations,
it follows that, if $i\ne j$, then $U_1(C_i)\cap U_1(C_j)$ is the socle of $P$ which is a simple
$\Lambda$-module.

Again from the relations of types one, two and three, we see that we have a short exact sequence of
$\Lambda$-modules 
\[
0\to \oplus_{i=1}^{r-1}S_v\stackrel{f}{\to} \oplus_{i=1}^{r}U_1(C_i)\to M\to 0,\] where
$f(s_1,\dots,s_{r-1})=(s_1,s_2-s_1,s_3-s_2,\dots,s_{r-1}-s_{r-2},-s_{r-1})$.  Factoring out
the socles, we obtain an   isomorphism.
\[  \oplus_{i=1}^{r}U_1(C_i)/S_v\stackrel{\cong}{\to} M/\soc(P).\]
Noting that   $M/ \soc(P)$  is the heart of  $ P$,
the proof is complete.
\end{proof}

\begin{Example}\rm{
Let $\Lambda$ be the \bca\ corresponding to our first example and let $P$ be the indecomposable
projective $\Lambda$-module associated to the vertex $1$ in the quiver of $\Lambda$  which in turn corresponds to the polygon $V_1$ of the associated Brauer configuration.
We give a schematic of $P$.
\vskip -.8in
\sloppy \hskip-.8in \includegraphics[scale=.8]{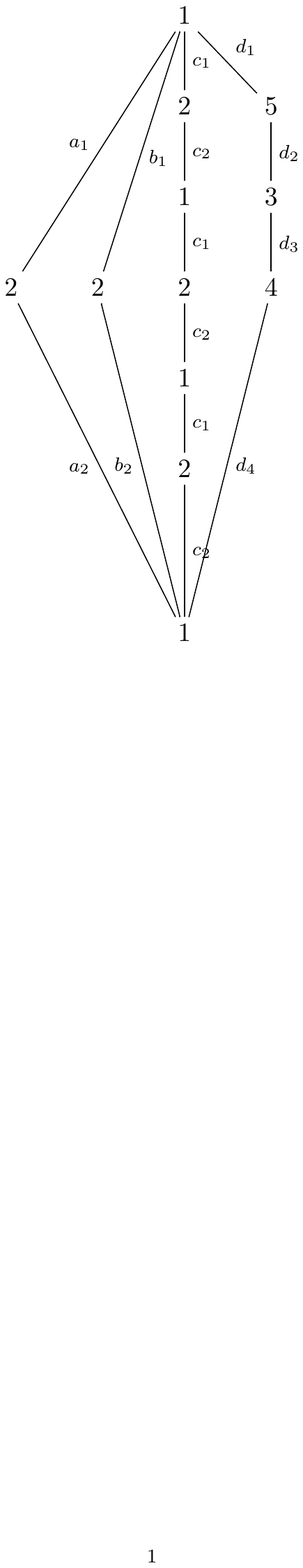}
}
\end{Example}
\vskip -4in
 In \cite{GreenSchroll}  the notion of a multiserial algebra is defined. Multiserial algebras are  a direct generalization of biserial algebras.  Namely, a $K$-algebra $A$ is 
\emph{multiserial} if the Jacobson radical of $A$  e- as a left and right $A$-module is a direct sum of uniserial modules, the
intersection of any two uniserial modules is either 0 or a simple $A$-module.
It follows
from Theorem~\ref{prop-proj} that a Brauer configuration algebra is multiserial and we also obtain the number of uniserial summands 
in the heart of each projective indecomposable:

\begin{cor}
Let $\Lambda$ be a \bca \ with \bc\ $\Gamma$.
 \begin{itemize}
 \item[(1)]  $\Lambda$ is a multiserial algebra.
 \item[(2)]  The number of summands in the heart of a projective indecomposable $\Lambda$-module $P$ such that $\rad^2(P)\ne 0$ equals the number of non-truncated vertices
 of the  polygon in $\Gamma$ corresponding to $P$  counting repetitions.
 \end{itemize}
\end{cor}

The next result uses Theorem \ref{prop-proj} and shows that the dimension of a \bca\  can easily be
computed from its \bc.

\begin{prop}\label{prop-dim-result}  Let $\Lambda$ be a \bca\ associated to the
\bc\ $\Gamma$
 and let $\cC = \{  C_1, \dots, C_t\}$ be a full set of
equivalence class representatives of special cycles. Assume that,
for $i=1,\dots,t$, $C_i$ is a special 
$\alpha_i$-cycle where $\alpha_i$ is a
nontruncated vertex in $\Gamma$.  Then  
\[ \dim_K \Lambda= 2{|Q_0|} + \sum_{C_i \in \cC}|C_i| (n_i|C_i|-1 ),
\]
where $|Q_0|$ denotes the number of vertices of $Q$,  $|C_i|$ denotes 
the number of arrows in the $\alpha_i$-cycle $C_i$ and $n_i =  \mu(\alpha_i)$.
\end{prop}

\begin{proof}
Since $\dim_K \Lambda=\dim_K(\Lambda/\rad(\Lambda))+\dim_K\rad(\Lambda)$ and
$\dim_K(\Lambda/\rad(\Lambda))=|Q_0|$ we must show that
$\dim_K\rad(\Lambda)=|Q_0| + \sum_{i=1}|C_i||C_i-1|^{\mu(\alpha_i)}$.
Now $\dim \rad(\Lambda)=\dim_K\soc(\rad(\Lambda))+\dim_K(\rad(\Lambda)/\soc(\rad(\Lambda)))=
|Q_0|+\dim_K (\rad(\Lambda)/\soc(\rad(\Lambda)))$.  We must show that $\rad(A)/\soc(\rad(A))= \sum_{i=1}|C_i||C_i-1|^{\mu(\alpha_i)}$ .
Using Theorem \ref{prop-proj}, we see that $\rad(\Lambda)/\soc(\Lambda)$ is isomorphic to the direct sum of the uniserial
modules $a \Lambda /\soc(a \Lambda )$, that is,
$\rad(\Lambda)/\soc(\Lambda) \simeq \oplus_{a\in Q_1}(a \Lambda /\soc(a \Lambda )$.  But if $C=a_1\cdots a_s$ is the special
$\alpha$-cycle with first arrow $a = a_1$ associated 
to the uniserial module $a \Lambda$, then $\dim_K(a \Lambda)=\mu(\alpha) s$.  	Hence,
$\dim_K(a \Lambda /\soc(a \Lambda ))=\mu(\alpha)|C|-1^{}$.  Noting there are
$|C|$ special cycles equivalent to $C$, the result now follows.   
\end{proof}   

\section{Radical cubed zero Brauer configuration algebras}\label{sec-r30}
In this section we classify the radical cubed zero \bca s such that
the associated \bc\ has no self-foldings; that is, each polygon in the \bc\ contains
no repeated vertices.
More precisely, we associate
\bca s to finite graphs  and compare them to the symmetric
radical cubed zero algebras associated to the same graphs. 
In particular, given a symmetric radical cubed zero algebra associated to a graph, we provide a construction
of a \bc, such that the associated \bca\ is isomorphic to the given symmetric radical cubed zero algebra.

We begin with
a well-known one-to-one correspondence.
Fix a positive integer $n$ and consider  the set
of finite graphs $G=(G_0,G_1)$ where $G_0=\{1,2,\dots, n\}$ is the set of vertices of $G$ and $G_1$ is the set of edges of $G$.
 We further suppose that $G$ has no isolated vertices; in particular, the valency of each
vertex is at least equal to one. We allow  multiple
edges and loops.  We say two such graphs
$G=(G_0,G_1)$ and  $G'=( G_0,G'_1)$  are equivalent  if there
is a set isomorphism $\delta\colon G_1\to G'_1$ such that, for each
$e\in G_1$, the endpoints of $e$ and $\delta(e)$ are the same.
Effectively, two graphs are equivalent if they differ only in the
names of the edges.

 Let $\cG_n$ denote the set of equivalence classes
of finite graphs having $n$ vertices labelled $1$ to $n$ and such that there are no vertices
of valency zero.

 Let $\cM_n$ denote the set of symmetric $n\times n$ matrices with
entries in the nonnegative integers such that no row only has $0$ entries.  Equivalently,
no column has only $0$ entries.  Although the next result is well-known we include
a short proof for completeness.

\begin{lemma}\label{lem-equiv}
There is a one-to-one correspondence between $\cG_n$ and $\cM_n$.
\end{lemma}
\begin{proof}
If $G$ represents an equivalence class in $\cG_n$, define $E(G)$ to be the $n\times n$ matrix with
$(i,j)$-entry being the number of edges with endpoints the vertices $i$ and
$j$,  if
$1\le i\ne j\le n$, and  the $(i,i)$-entry of $E(G)$
is the number of loops at vertex $i$.

If $E=(e_{i,j})\in\cM_n$ let $G(E)$ be the equivalence class of a graph with vertex set $\{1,\dots,n\}$
and $e_{i,j}$ edges with endpoints $i$ and $j$.

We have $G(E(G))=G$ and $E(G(E))=E$.  Finally, $G$ having no isolated vertices corresponds
to $E$ having no zero row and no zero column and we are done.
\end{proof}
 Our goal is to
extend the one-to-one correspondence given in the above Lemma to include \bc s.  For this we need to introduce ordered \bc s.  But before we do so, we would like to motivate why this is necessary.
Consider the following two (unequal) symmetric matrices
$$\left(\begin{array}{cccc} 0&1&0&1\\1&0&1&0\\ 0&1&0&1\\1&0&1&0\end{array}\right)
\mbox{ and } \left(\begin{array}{cccc} 0&0&1&1\\0&0&1&1\\ 1&1&0&0\\1&1&0&0\end{array}\right).$$
The graphs associated to these matrices are both 4 cycles respectively given  by

$$\xymatrix{
1 \ar@{-}[d] \ar@{-}[r] & 2\ar@{-}[d]  &              &      1 \ar@{-}[d] \ar@{-}[r] & 3\ar@{-}[d]     \\
4 \ar@{-}[r] &3                         & \mbox{ and }  &
4 \ar@{-}[r] &2.} $$

As the vertices of the graph will correspond to the
polygons in a \bc, to obtain the desired one-to-one correspondence,  one needs to distinguish two \bc s that differ only in the labelling of the polygons.  This
is taken care of by `ordering' the polygons, which is formally defined below.

Before giving the definition of ordering, we have the following result.

\begin{lemma}
Let $\Lambda$ be a indecomposable \bca\ with  \bc\ $\Gamma$. Then $\rad^3(\Lambda)=0$ and
$\rad^2(\Lambda)\ne 0$ if and only if  $\val(\alpha)\mu(\alpha)=2$,
for each nontruncated vertex $\alpha\in\Gamma_0$.
\end{lemma}

\begin{proof}
If $\alpha$ is a nontruncated vertex, then $\val(\alpha)\mu(\alpha)\ge 2$.  If, for some $\alpha\in\Gamma_0$,
$\val(\alpha)\mu(\alpha)\ge 3$ and  $C$ is a special $\alpha$-cycle, then
$C^{\mu(\alpha)}$ has length $\val(\alpha)\mu(\alpha)$ and is a nonzero
element in $\soc(\Lambda)$.  Thus $\rad^3(\Lambda)\ne 0$.  Conversely, if $\rad^3(\Lambda)\ne 0$ then there must be some vertex $\alpha$ and
a special $\alpha$-cycle $C$ such that the 
length of $C^{\mu(\alpha)}$ which equals $\val(\alpha)\mu(\alpha)$ is $\ge 3$.
To see this, if every $C^{\mu(\alpha)}$ is of length $2$
then every path of length 3 must be in the ideal of relations.  In particular,
$\rad^3(\Lambda)=0$, a contradiction. 
\end{proof}


  Let $\Gamma=	(\Gamma_0,\Gamma_1,\mu,\mathfrak o)$ be a reduced \bc \ such that
$|\Gamma_1| = n$. That is, there are
exactly $n$ polygons in $\Gamma_1$, and suppose further that
 $1\le\val(\alpha)\mu(\alpha)\le 2$, for all  $\alpha\in\Gamma_0$.
Recall that by our assumptions on \bc s and the definition of reduced,  if $\val(\alpha)\mu(\alpha)=1$, then
$\alpha$ is a truncated vertex in a  $2$-gon and the other vertex is not truncated.  

 We say $\Gamma$ is \emph{ordered by $f$} if
$f\colon \{1,2,3,\dots,n\}\to \Gamma_1$ is an isomorphism.  Note that $f$
`orders' the polygons in $\Gamma_1$ with $f(i)$ being the $i^{th}$ polygon.
To simplify notation, we will usually omit $f$ and use the terminology `the $i^{th}$
polygon' for $f(i)$ and say $\Gamma$ is \emph{ordered}.
We remark
that the assumption that  $\val(\alpha)\mu(\alpha)\le 2$ for all $\alpha$, implies
there is only one choice for $\mathfrak o$ and $\mu$ is determined by
the truncated vertices.    Namely, if $\alpha$ is a vertex
with $\val(\alpha)=2$ and $\mu(\alpha)=1$ and if $\alpha$ is in the
polygons $V$ and $V'$, with $V\ne V'$, then
the successor sequence at $\alpha$ must be $V<V'$ which is the same as $V'<V$.
If $\alpha$ is a vertex with $\val(\alpha)=1$,
$\alpha$ is either truncated or not.  If $\alpha$ is truncated,
$\mu(\alpha)=1$ and if not, $\mu(\alpha)=2$, and   $\alpha\in V$, then the successor sequence for $\alpha$
is given by $V$.   If $\alpha$ is truncated
then $\alpha$  has no successor sequence.

We say an ordered reduced \bc\  $\Gamma$ has \emph{no self-foldings} if, for each polygon $V$
in  $\Gamma_1$, there are no repeated vertices in $V$.

  Let   $\Gamma=	(\Gamma_0,\Gamma_1,\mu,\mathfrak o)$
and  $\Gamma'=	(\Gamma'_0,\Gamma'_1,\mu',\mathfrak o')$
be two ordered  reduced \bc s with no self-foldings such that, for each $\alpha \in \Gamma_0$ or $\alpha \in \Gamma'_0$,
we have $1\le\val(\alpha)\mu(\alpha)\le 2$.
We say $\Gamma$ is \emph{equivalent} to $\Gamma'$
if there is a set isomorphism $\epsilon\colon\Gamma_0\to\Gamma'_0$ such that
if $\{\alpha_1,\dots,\alpha_r\}$ is the $i^{th}$  polygon in $\Gamma_1$, then $\{\epsilon(\alpha_1),\dots,\epsilon(\alpha_r)\}$ is the $i^{th}$  polygon
in $\Gamma'_1$.
Effectively, two ordered reduced \bc s are equivalent  if they differ only in the
names of the vertices.

 It easily follows from the definition, that if $\Gamma$ and $\Gamma'$ are equivalent  \bc s then
the associated \bca s are isomorphic. This holds since the only
difference between $\Gamma$ and $\Gamma'$ is the labeling of the
vertices.

Let $\cB_n$ be the set of equivalence classes of ordered reduced \bc s $\Gamma$ satisfying
\begin{enumerate}
\item $\Gamma$ has no self-foldings,
\item $\Gamma$ has exactly $n$ polygons, and
\item    $1\le \val(\alpha)\mu(\alpha)\le 2$, for all  $\alpha\in\Gamma_0$.
\end{enumerate}

Before stating our next result, we introduce some terminology.
Let $G= (G_0,G_1)$ be a representative of an element of $\cG_n$.
A vertex in $G$ is called a \emph{leaf} if it has valency $1$.
If $i$ is a leaf and $i$ is an endpoint of the edge $e$,  we say
$e$ is the \emph{leaf edge associated to $i$}.

\begin{prop}\label{prop-equiv}
The set $\cB_n$ is in one-to-one correspondence
with the sets $\cG_n$ and $\cM_n$, for all strictly positive integers $n$.
\end{prop}
\begin{proof} We show that $\cB_n$ is in one-to-one correspondence with $\cG_n$.

Let $\Gamma=(\Gamma_0,\Gamma_1,\mathfrak o,\mu)$ be an ordered,
reduced \bc\ which represents
an equivalence class in $\cB_n$.
 We construct a graph  $G(\Gamma)$ which represents an equivalence class in $\cG_n$ as follows. For
$i=1,\dots n$, let $V_i$ be the
$i^{th}$ polygon in $\Gamma_1$.  The vertex set of $G(\Gamma)$ is given by $\{1,\dots, n\}$ with the polygon $V_i$ in $\Gamma$ corresponding
to the vertex $i$ in $G(\Gamma)$.
The number  of edges between vertices $i$ and $j$ in $G(\Gamma)$, where $i\ne  j$, is equal to the number
of vertices $\alpha\in\Gamma_0$ such that  $\alpha\in V_i\cap V_j$.
The number of loops at vertex $i$
in $G(\Gamma)$ is equal to the number of vertices $\alpha\in\Gamma_0$ such that  $\alpha\in V_i$,
and  $ \mu(\alpha)=2$.   From $\Gamma$, we have constructed
a graph $G(\Gamma)$.     The reader may check that if $\Gamma'$ is
another \bc\ in the same equivalence class of $\cB_n$ as $\Gamma$, then
$G(\Gamma)$ and $G(\Gamma')$ are in the same equivalence class in $\cG_n$.
Thus, we have a well-defined map from $\cB_n$ to $\cG_n$.

We will construct the inverse  map from $\cG_n$ to  $\cB_n$.   We begin with
a special case.  Let  $G$ be the graph with one vertex  $\{1\} $ and one loop
at  $ 1$.  Set  $\Gamma(G)=(\{\alpha,\beta\}, V=\{\alpha,\beta\},\mu,\mathfrak o)$
where $\mu(\alpha)=1$ and $\mu(\beta)=2$.  Note that there is only one
 orientation since
there is just one polygon.  The vertex $\alpha$ is truncated. Then we send the equivalence class of $G$ to the equivalence
class of $\Gamma(G)$.  It is clear that $G(\Gamma(G))$ is in the same equivalence
class as $G$ in  $\cG_n$.

Now let $G = (G_0, G_1)$ be a graph in an equivalence class in $\cG_n$.
Using the special case above, we may assume that   $G$ has no connected
components consisting of a vertex and a loop.
Define a \bc\  $\Gamma(G)
=(\Gamma_0,\Gamma_1,\mu, \mathfrak o)$  representing an equivalence class in $\cB_n$ as follows. 
The vertex set of $\Gamma(G)$ is 
$$\Gamma(G)_0 = G_1\cup \{(e,i)\mid i \text{ is a leaf and }e\text{ is the leaf-edge
associated to }i\}.$$  The set of polygons $\{V_1,\dots, V_n\}$  is in bijection with the set $G_0$ of vertices of $G$,
where $V_i=\{e\in G_1\mid i \text{ is an endpoint of }e\}$ if $i$ is
not a leaf and $V_i=\{e, (e,i)\}$ where $i$ is a leaf and $e$ is the leaf edge
associated to $i$. If $e$ is a loop at vertex $i$ in $G$, then
$V_i$ contains $e$ as an element once  (since
we are constructing \bc s with no self-foldings;  as an aside,   Example \ref{ex-r30} shows that if we allow
self-foldings, there could be more than one Brauer configuration associated to a graph $G$).
Note that the valency of a vertex in $\Gamma(G)$
is at most 2 since an edge in $G$ has at most two endpoints.  Therefore there are at most two polygons at any vertex of $\Gamma(G)$ and there is a unique cyclic order
$\mathfrak o$.  To clarify notation, if $e$ is an edge in
$G$ then we will write $\bar e$ instead of $e$ for the vertex in $\Gamma$.
We define the multiplicity function as follows. If $e$ is an edge of $G$
and $e$ is not a loop, then $\mu(\bar e)=1$. If $e$ is a loop,
$\mu(\bar e)=2$.  Finally set $\mu(e,i)=1$.   We see that the
vertices of $\Gamma(G)$ of the form $(e,i)$ are precisely the truncated vertices
of $\Gamma(G)$.
Moreover, it is clear that $\Gamma(G)$ is ordered and reduced.
 Thus we have constructed a \bc\
$\Gamma(G)$.  The reader may check that sending the equivalence class of $G $ to the equivalence
class of $\Gamma(G)$ is a well-defined map and inverse to the first map.
\end{proof}

For a given graph $G$ in an equivalence class of $\cG_n$, call the \bc s  constructed in the proof of Proposition~\ref{prop-equiv} above,
  the \emph{Brauer
configuration   associated to the
graph $G$} and denote it by $\Gamma(G)$.

\begin{Example}\label{ex-gph-bc}{\rm
 To help clarify the above proof we provide examples of both constructions. Namely in the first instance, given a Brauer configuration we construct a representative of the associated
graph equivalence class. Secondly, given a graph, we construct a representative of the corresponding Brauer configuration.

a) Let   $\Gamma=( \Gamma_0,
\Gamma_1,\mu,\mathfrak o)$ be a \bc\ where $\Gamma_0=\{
1,2,\dots, 10\}$, $\Gamma_1=\{V_1,\dots, V_6\}$, with $V_1=\{1,2,3\},  V_2=\{3,4,5\},
V_3=\{1,5\}, V_4=\{2,4,6,7,9\}, V_5=\{7,8\},$ and $V_6=\{9,10\}$, $\mu(8)=2 =\mu(6)$, and $\mu(i)=1$, for $i\ne 6, 8$.  Note that since at each vertex of $\Gamma$ the valency is
$\le 2$, there is a unique orientation  $\mathfrak o$.
We further remark that
neither vertices 6 nor 8 are truncated since their multiplicities are greater than 1, however,  vertex 10  is
truncated.  A realization of
$\Gamma$ is given by 

\begin{center}
\includegraphics[scale=.5]{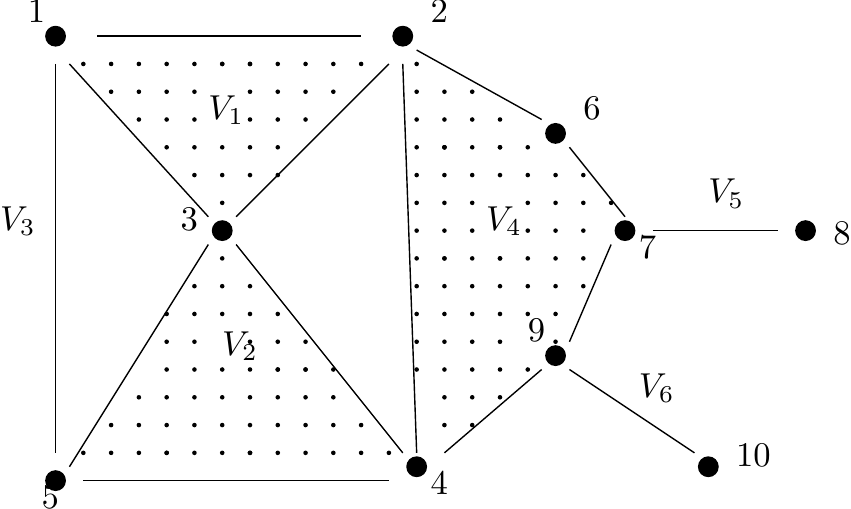}
\end{center}

The reader may check that the equivalence class of the graph $G(\Gamma)$ as constructed in the theorem  can
be represented by 
\begin{center}\includegraphics[scale=.5]{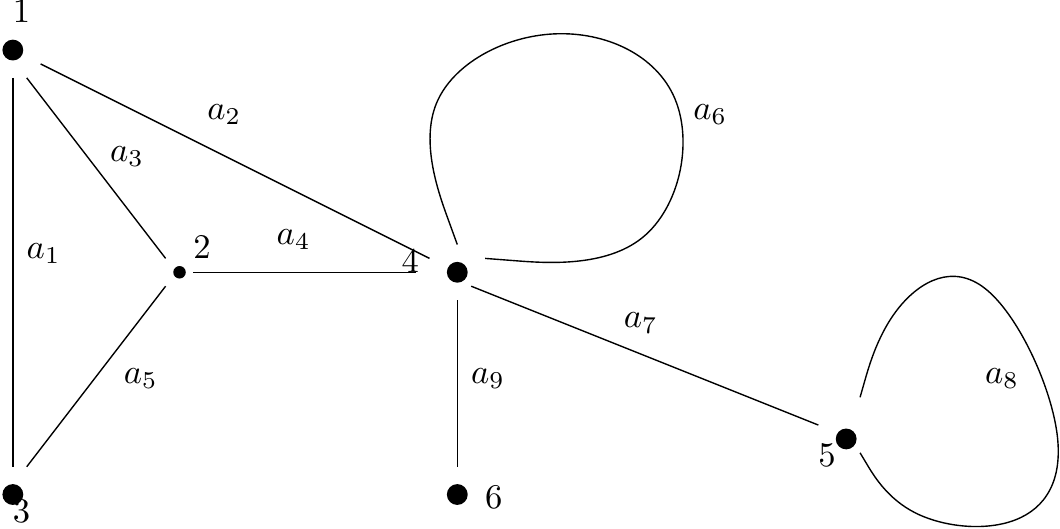}
\end{center}


 Note that the edges are
labelled $a_i$, where $i$ is the vertex in  $\Gamma$ yeilding $a_i$.  Thus, since vertex 1
in $\Gamma$ is in $V_ 1 \cap V_3$, we have an edge $a_1$ between $1$ and $3$.
Similarly, for example, since vertices 6 and 8 in $\Gamma$ has multiplicity 2, we get
the loops $a_6$ and $a_8$.  Note that since $10$ is truncated in $\Gamma$, there
is no $a_{10}$ and vertex 6 of $G(\Gamma)$ is a leaf.

As a second example, we start with a graph $G$ representing an equivalence class
in $\cG_6$ 
\begin{center}
\includegraphics[scale=.5]{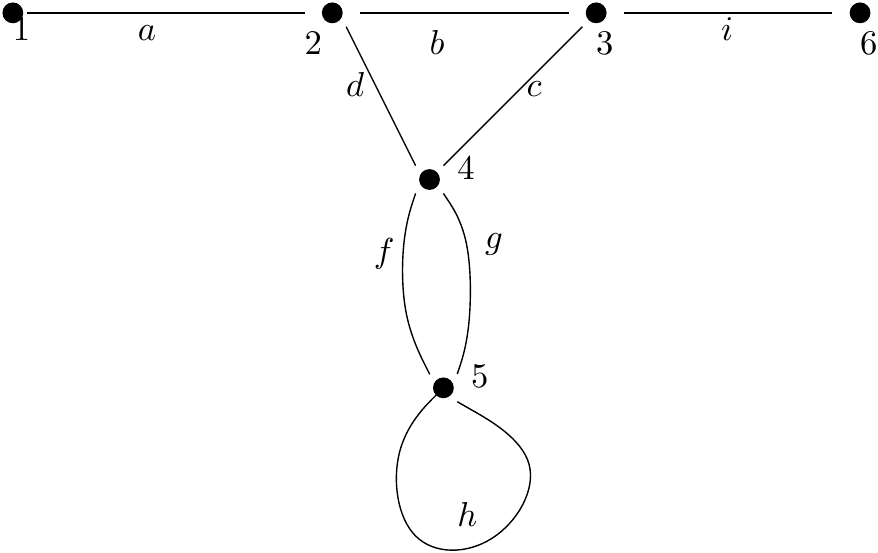}
\end{center}

and construct the associated \bc, which is

\begin{center}
\includegraphics[scale=.8]{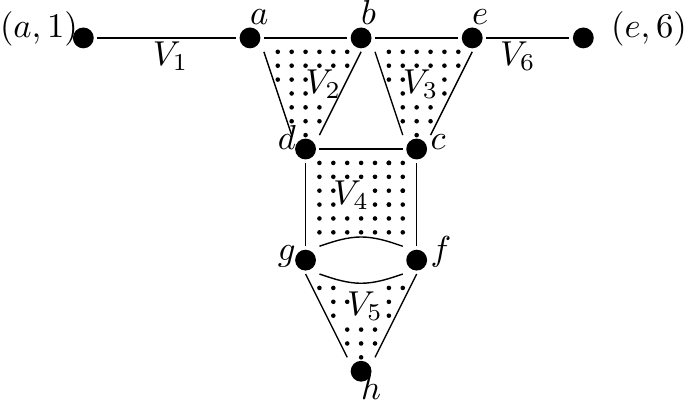}
\end{center}

Note that the two leaves at vertices $1$ and $6$ in $G$, result in two truncated
vertices in $\Gamma(G)$; namely $(a,1)$ and $(e,6)$,  respectively.
}
\end{Example}

Given a graph $G$ representing an equivalence
class in $\cG_n$, there is a standard construction of a radical cubed
zero symmetric algebra which we now recall.
 We start by constructing  a quiver, $\cQ_G$, associated to $G$ as
follows.  The vertex set of  $\cQ_G$ consists of  $n$ vertices labeled $\{1,\dots, n\}$ and, for each edge $e$
in $G_1$ that is not a loop, and such that $e$ has endpoint vertices $i$ and $j$ in $G_0$,
there are two arrows $a$ and $a'$ in $\cQ_G$
such that  $a$ is an arrow from  $i$
to $j$ and $a'$ an
an arrow from $j$ to $i$.  For each loop $e$ at vertex $i$ in  $G$, the quiver $\cQ_G$ has a loop $b$
at vertex
$i$.

We call $aa'$, $a'a$, and $b^2$
\emph{distinguished paths} in $\cQ_G$.

 Let $I_G$ be the ideal in $K\cQ_G$ generated by
\begin{enumerate}
\item  all $p-q$ where $p$ and $q$ are distinguished path with the
same start and end vertices.
\item all $aba$,
for all arrows $a, b$ in $\cQ$ such that $ab$ is a distinguished path.
\item all non-distinguished paths of length $2$.
\end{enumerate}

 Note that if $G'$ is a graph equivalent to $G$ then   $\cQ_G =\cQ_{G'}$  and
$I_G = I_{G'}$.

Keeping the notation above, we have the following result.

\begin{lemma}\label{lem-r3} Let  $G $ be a graph
representing  an equivalence class in $\cG_n$. Then  $K\cQ_G/I_G$ is a symmetric
$K$-algebra with radical cubed zero.
\end{lemma}

\begin{proof} Let $\pi\colon K\cQ_G\to K\cQ_G/I_G$ be the
canonical surjection.
It is  routine to check that $K\cQ_G/I_G$ is a selfinjective
algebra with radical cubed zero.

To see that $K\cQ_G/I_G$ is symmetric,
define $f\colon K\cQ_G /I_G \to K$
 as follows. If $p$ is a path in $K\cQ_G$, let
\[
f(\pi(p))= \begin{cases} 
1,&\text{if }p\text{ is a distinguished path},\\
0,&\text{in all other cases.}
\end{cases}
\]
It is straightforward to show that the kernel
of $f$ contains no two-sided ideals and that $f(xy)=f(yx)$,
for all $x, y\in K\cQ_G/I_G$.
\end{proof}

We call $A_G = K\cQ_G/I_G$ the \emph{canonical symmetric algebra with
radical cubed zero associated
to $G$}.	

We now present the main result of this section.

\begin{thm}\label{thm-rad3}
 Let $G$ be a representative of an equivalence class
in $\cG_n$ and let $\Gamma(G)$  be the \bc\ associated to $G$ with corresponding \bca\ $\Lambda_{\Gamma(G)}$. Furthermore, let $A_G$ be the canonical symmetric algebra
with radical cubed zero associated to $G$.  Then    $\Lambda_{\Gamma(G)}$
 and  $A_G$ are isomorphic as $K$-algebras.
\end{thm}

\begin{proof}
For ease of notation set $\Gamma =  \Gamma(G)$,  $\Lambda = \Lambda_{\Gamma(G)}$ and $A= A_G$.
  The quiver $\cQ_{\Lambda}$ of $\Lambda$ has vertex set $\{v_1,\dots,v_n \}$ corresponding
to the set of polygons $\{V_1,\dots,V_n\}$.  Recall that the vertex set $\Gamma_0$ of $\Gamma$ is given by the edge set
$G_1$ of $G$ together with the truncated vertices $(e,i)$, where $i$ is a leaf and
$e$ is the leaf edge associated to $i$.

For the special case where $G$ has one vertex and  one loop at that vertex,
both $\Lambda$ and $A$ are isomorphic to $K[x]/(x^3)$  since  the \bc\ associated
to $G$ is given in the proof of Proposition \ref{prop-equiv} is a 2-gon with one
vertex truncated and the other has multiplicity 2.

Thus we may assume that $G$ has no connected component consisting of one vertex and
 one loop.
 Let $e\in\Gamma_0$ where $e$ is an edge in $G$. To avoid confusion, we will
write $\bar e$ when $e$ is a vertex in $\Gamma$ and $e$ when $e$ is an
edge in $ G$.  Then, as an edge in $G$, $e$ either is a loop or not.  If $e$ is a loop in $G$, then let $i$ be the vertex in $G_0$ that is the endpoint of $e$.  From our construction it then follows that, as a vertex in $\Gamma_0$,
$\bar e$ has valence $1$ in $\Gamma$, that $\bar e\in V_i$ and $\mu(\bar e)=2$.
But then in ${\Gamma}$, the  successor sequence at $\bar e$ is $V_i$ and it gives rise to a loop
at $v_i$ in $\cQ_{\Lambda}$.   On the other hand, if $e$ is not a loop in $G$, then $e$ has two endpoints $i$ and $j$, with  
$1\le i \neq   j\le n$.   It follows from our construction of $\Gamma$,  that $\bar e$ as a vertex in $\Gamma_0$ is such that
$\val(\bar e)=2$ with $\bar e\in V_i$ and $\bar e\in V_j$.  Furthermore, the successor sequence at $\bar e$ is
$V_i<V_j$.  It follows that $\bar e$  gives rise to two arrows in $\cQ_{\Lambda}$, namely one arrow from $i$ to $j$ and one from $j$ to $i$.  Recall that the truncated vertices of $\Gamma$
do not influence $\cQ_{\Lambda}$.

Let $\cQ_A$ be the quiver of the canonical symmetric algebra with radical cubed zero  $A$ associated to $G$.
 It clearly follows from the construction of $\cQ_\Lambda$ and $\cQ_G$  that by sending $v_i$ to $i$  and an arrow from $v_i$ to $v_j$ to an arrow from $i$ to $j$, we obtain an
isomorphism  $\cQ_{\Lambda}\to \cQ_G$.
Let $I_{\Gamma}$ be the ideal of relations of the \bca\ associated to $\Gamma$
and $I_G$ be the ideal of relations of   $A$.
It easy to see that, under the isomorphism  $\cQ_{\Lambda}\to \cQ_A$, the type one, two and three relations obtained from
the \bc\ $\Gamma$  map isomorphically
to the relations (1), (2), and (3) for  the canonical symmetric algebra with radical cubed zero associated to $G$
given above. This completes the proof.

\end{proof}

One might ask if every radical cubed zero symmetric algebra  is the canonical symmetric algebra with radical cubed zero associated to a graph $G$.  The answer is no as shown by
the following example  of a graph to which we can associate two non-isomorphic symmetric algebras with radical cubed zero.

\begin{Example}\label{ex-r30}{\rm
We give an example of a graph $G$ in $\mathcal G_1$ and two symmetric algebras  associated to this graph, one of them being the canonical symmetric algebra with radical cube zero associated to $G$, the other one being a Brauer configuration algebra where the Brauer configuration has a self-folding.  These two algebras are isomorphic
if $K$ contains an element $i$ such that $i^2= -1$ and they are non-isomorphic otherwise.

Let $G$ be the graph with one vertex and two loops.

The
canonical symmetric algebra associated to  $G$ is
$A=K[x,y]/(xy, x^2-y^2)$, where $K$ is a field and
$K[x,y]$ is the commutative polynomial ring in two
variables.   A \bc\ such that the associated \bca\
is isomorphic to $A$ is $\Gamma=(\{\alpha,\beta\}, V=\{\alpha,
\beta\},\mu,\mathfrak o)$ with $\mu(\alpha)=\mu(\beta)=2$.  In fact, this corresponds to the Brauer graph algebra associated to the
Brauer graph given by a single edge where both vertices have multiplicity two.

Now consider the algebra $A'=K[x,y]/(x^2,y^2)$.  Then
$A'$ is a symmetric radical cubed zero algebra.   The algebra $A'$
is isomorphic to the \bca\ associated to the \bc\
$\Gamma'=(\{\alpha\}, V=\{\alpha,\alpha\},\mu, \mathfrak o  )$
with $\mu(\alpha)=1$.  Remark that $\Gamma'$ has
 a self-folding.  The algebra $A'$ also corresponds
 to a Brauer graph algebra.
Its  Brauer graph
is given by a loop with multiplicity one at the only vertex.

Note that the matrix associated to both $A$ and $A'$ is
the $1\times 1$ matrix $(2)$.

 Now suppose that $i \in K$ with $i^2 =-1$. Then the map defined by
$x \mapsto x +iy,  y \mapsto x-iy$ induces an isomorphism from $A'$ to $A$. If, however, $K$ does not contain an element squaring to $ -1$ then the algebras $A$ and $A'$ defined above are not isomorphic.}
\end{Example}
\vskip .1in

 Example~\ref{ex-r30} leads to the question:
What can one say about an indecomposable symmetric
radical cubed zero algebra of the form $K\cQ/I$.  This question
is addressed  in \cite{GreenSchroll} where it  is shown that  every such
algebra is a \bca, when one allows self-foldings in the Brauer configuration.

\bibliographystyle{plain}

\end{document}